\documentclass[final,leqno,onefignum,onetabnum]{siamltex1213}
\usepackage[linesnumbered,ruled,vlined]{algorithm2e}
\usepackage{amsfonts}
\usepackage{amsmath}
\usepackage{array}
\usepackage{cases}
\usepackage{color}
\usepackage{subfigure}

\newcommand{\lvertiii}{{\left\vert\kern-0.25ex\left\vert\kern-0.25ex\left\vert}}
\newcommand{\rvertiii}{{\right\vert\kern-0.25ex\right\vert\kern-0.25ex\right\vert}}
\newtheorem{assume}[theorem]{Assumption}
\newtheorem{remark}[theorem]{Remark}

\title{An Adaptive Step Size Strategy for Orthogonality Constrained Line Search Methods
\thanks{This work was supported by the National Science
Foundation of China under grants 91730302 and 11671389 and the Key Research Program of Frontier Sciences of the Chinese Academy of Sciences under grant QYZDJ-SSW-SYS010.}}

\author{Xiaoying Dai\footnotemark[2], Liwei Zhang\footnotemark[2],  and Aihui Zhou\footnotemark[2]}

\begin{document}

\maketitle

\renewcommand{\thefootnote}{\fnsymbol{footnote}}

\footnotetext[2]{LSEC, Institute of Computational Mathematics and Scientific/Engineering Computing,
Academy of Mathematics and Systems Science, Chinese Academy of Sciences,  Beijing 100190, China; and School of Mathematical Sciences,
University of Chinese Academy of Sciences, Beijing 100049, China. \{daixy, zhanglw, azhou\}@lsec.cc.ac.cn}

\author{
Xiaoying Dai\thanks{LSEC, Institute of Computational Mathematics
and Scientific/Engineering Computing, Academy of Mathematics and
Systems Science, Chinese Academy of Sciences, Beijing 100190, China
(daixy@lsec.cc.ac.cn).} \and  Liwei Zhang\thanks{LSEC, Institute of Computational Mathematics
and Scientific/Engineering Computing, Academy of Mathematics and
Systems Science, Chinese Academy of Sciences, Beijing 100190, China
(zhanglw@lsec.cc.ac.cn).}
\and Aihui Zhou\thanks{LSEC, Institute of Computational Mathematics
and Scientific/Engineering Computing, Academy of Mathematics and
Systems Science, Chinese Academy of Sciences, Beijing 100190, China
(azhou@lsec.cc.ac.cn).}
}


\begin{abstract}
  In this paper, we propose an adaptive step size strategy for a class of line search methods for orthogonality constrained minimization problems, which avoids the classic backtracking procedure.
  We prove the convergence of the line search methods equipped with our adaptive step size strategy under some mild assumptions. We then apply the adaptive algorithm to electronic structure calculations. The numerical results show that our strategy is efficient and recommended.
\end{abstract}




\section{Introduction}
Orthogonality constrained minimization problem
\begin{equation}\label{general-form}
\min_{U\in \mathcal{M}^N} \ \ \ E(U),
\end{equation}
is a typical model in modern scientific and engineering computing, including the extreme
eigenvalue problem \cite{Golub,Sa,Smith94}, the low-rank correlation matrix
problem \cite{HMWY,Van}, the leakage interference minimization problem \cite{LDL}, and the
Kohn-Sham Density Functional Theory(DFT) in electronic structure calculations
\cite{DLZZ,HMWY,Ma,SRNB,ZZWZ}. Here $E(U)$ is an energy functional on a Stiefel manifold $$\mathcal{M}^N = \{U\in V^N: U^TU = I_N\}$$ with $N\ge1$.

We see that the line search method is the most direct way to solve \eqref{general-form}
and has been widely investigated. In particular, the line search method
has been applied to orthogonality
constrained problems (see, e.g., the gradient type method
\cite{JD,SRNB,WY,ZZWZ}, the conjugate gradient(CG) method \cite{DLZZ,EAS},
and the Newton type method \cite{DZZ,EAS,HMWY,ZBJ}). We refer to \cite{AbMaSe, Smith94} for
the constrained line search method on an abstract manifold and \cite{AbBaGa, GLY, YMW}  for some other methods apart from line search methods such as trust-region methods and a parallelizable infeasible methods for manifold constrained optimazation.

We understand that the step size strategy plays a
crucial rule in a line search method.  Since the computational cost of the exact line search is usually unaffordable,
the ``Armijo backtracking" approach proposed in \cite{LA} is performed as
an alternative way that leads to some monotone algorithms for orthogonality constrained
problems\cite{AbMaSe,DLZZ,ZBJ}. The non-monotone step size strategies based on the similar
``Armijo-type backtracking" approaches are presented in order to accelerate the line search methods
\cite{DZ,DYH,ZH}. The effectiveness of the non-monotone step sizes remains well when applied
to minimization problems with orthogonality constraints \cite{JD,ZZWZ}.
To our knowledge, most of the existing line search algorithms for solving manifold
constrained problems require the backtracking skill to ensure the convergence.
However, these backtracking-based step size strategies need to compute the trial points
and their corresponding function values 
repeatedly if they do not meet the Armijo-type condition.
During this procedure, not only are the times of backtracking unpredictable(which usually
means that the final step size is unassessable), much computational cost is also needed, 
especially for the orthogonality constrained problem, where the orthogonalization procedure
is required.

To reduce the computational cost in finding reasonable step sizes, in this paper, we propose and
analyze an adaptive step size strategy for a class of
line search methods for orthogonality constrained minimization problems.
It is shown by theory and numerics that we are able to avoid the classic ``backtracking" approach in the line search method without losing the convergence.
As an application, we apply our adaptive strategy to solve the Kohn-Sham energy minimization problem,  which is a significant and challenging scientific model, for several
typical systems. The numerical results show that our approach indeed outperforms the backtracking-based step size
strategies in both number of iterations and computational time.

The rest of this paper is organized as follows: in Section \ref{sec-pre}, we provide a brief
introduction to the orthogonal constrained minimization problems and some notation that
will be used in this paper. We set up an uniform framework for a class of line search methods for orthogonality constraints minimization problems and review the classic ``backtracking-based" step size strategy before we study the adaptive step size strategy.
In Section \ref{sec-asss}, we propose our adaptive step size strategy and prove the convergence
of the corresponding line search methods. We report several numerical experiments on electronic
structure calculations in Section \ref{sec-num} to show the effectiveness and advantages of our 
strategy. Finally, we give some concluding remarks in Section \ref{sec-cln}, provide the
proof and remarks to Theorem \ref{btbb} in Appendix \ref{A1}, and some numerical results obtained by the gradient
type method with different initial step size choices and different parameter in Appendix \ref{A2}.

\section{Preliminary}\label{sec-pre}
\subsection{Setting} \label{subsec-setting}
Let $U = (u_1, \dots, u_N), \ W = (w_1, \dots, w_N)\in V^N$, where $V$ is some Hilbert space
equipped with the inner product $\langle\cdot, \cdot\rangle_V$. Denote
$U^TW = (\langle u_i, w_j\rangle_V)_{i,j=1}^N$ the inner product matrix of $U$ and $W$.
We deduce a inner product in $V^N$ as $\langle U, W\rangle_{V^N}=\text{tr}(U^TW)$ and
define the induced norm of $V^N$ by $\|U\|_{V^N}=\sqrt{\langle U, U\rangle_{V^N}}$.

Consider minimization problem:
\begin{equation}\label{dis-emin}
\begin{split}
&\inf_{U\in V^N} \ \ \ E(U) \\
& s.t.\ \   U^TU = I_N,
\end{split}
\end{equation}
where $I_N$ is the identity matrix of order $N$. The feasible set of \eqref{dis-emin} is a Stiefel manifold which is defined as
\begin{equation}
\mathcal{M}^N = \{U\in V^N: U^TU = I_N\}.
\end{equation}

In this paper, we mainly focus on the objective functional that is orthogonal invariant, namely,
\begin{eqnarray}\label{orthogonal-invariant}
E(U)=E(UP), ~~\forall U\in \mathcal{M}^N,P\in \mathcal{O}^N,
\end{eqnarray}
where $\mathcal{O}^N$ is the set of all orthogonal matrix of order $N$. 
We should point out that for $U\in V^N, \ P\in\mathbb{R}^{N\times N} $, product $UP\in V^N$ can be viewed as the vector-matrix product since $U = (u_1, u_2, \cdots, u_n)$ with $u_i\in V$ can be viewed as a $1\times N$ vector. Under the orthogonal invariant setting \eqref{orthogonal-invariant}, we may consider \eqref{dis-emin} on a Grassmann manifold $\mathcal{G}^N$ which is the quotient manifold of $\mathcal{M}^N$:
\begin{equation*}
\mathcal{G}^N = \mathcal{M}^N/\sim.
\end{equation*}
Here, $\sim$ denotes the equivalence relation which is defined as: $\hat{U} \sim U$,  if and only if there exists $P \in \mathcal{O}^N$, such that $\hat{U} = UP$. For any $U \in \mathcal{M}^N$, we denote
\begin{equation*}
[U] = \{U P: P \in  \mathcal{O}^N\},
\end{equation*}
and Grassmann manifold $\mathcal{G}^N$ is then formulated as
\begin{equation*}
\mathcal{G}^N= \{[U]: U\in\mathcal{M}^N\}.
\end{equation*}
In addition, we assume that \eqref{dis-emin} achieves its minimum in $\mathcal{G}^N$, which implies that
\eqref{dis-emin} is equivalent to
\begin{equation}\label{Emin-Grass}
\min_{[U] \in \mathcal{G}^N} \ \ \ E(U).
\end{equation}
For $[U] \in \mathcal{G}^N$, the tangent space of $[U]$ on $\mathcal{G}^N$ is the following set \cite{EAS}
\begin{equation}\label{tan}
\mathcal{T}_{[U]}\mathcal{G}^N = \{W\in V^N:
W^TU   = {\bf 0} \in \mathbb{R}^{N\times N}\}.
\end{equation}
The union of all tangent spaces is called the tangent bundle, which is denoted by
\begin{equation*}
\mathcal{T}\mathcal{G}^N = \bigcup_{[U]\in\mathcal{G}^N}\mathcal{T}_{[U]}\mathcal{G}^N.
\end{equation*}
Further, the gradient $\nabla_G E(U)$ at
$[U]$ on $\mathcal{G}^N$ is \cite{EAS}
\begin{equation*}
    \nabla_G E(U)  = \nabla E(U) - U\big(U^T\nabla E(U)\big) = (I - UU^T)\nabla E(U),
\end{equation*}
where $\nabla E(U)$ is the classic gradient of $E$ at point $U$ and $I$ is the identity in $V^N$. Note that $\nabla_G E(U)\in\mathcal{T}_{[U]}\mathcal{G}^N$.
\subsection{Orthogonality constrained line search method}\label{subsec-lsm}
For solving \eqref{Emin-Grass}, a direct approach is to use the so called line search method, such as gradient type method, Newton method, and CG method.
In this part, we set up an uniform framework for a class of line search methods with orthogonality constraints.

Suppose $U\in\mathcal{M}^N$ is our current iteration point. There are two main issues in a line search method, the search direction $D\in\mathcal{T}_{[U]}\mathcal{G}^N$ and the step size $t\in\mathbb{R}$. After these two issues are handled, we need to apply an orthogonality preserving operator in a feasible method to ensure that the next iteration point is still in the feasible set. To this end, the so called ``retraction" is used \cite{AbMaSe}.
\renewcommand{\thefootnote}{\arabic{footnote}}
\setcounter{footnote}{0}

Given an operator $\kappa:\mathcal{T}_U\mathcal{M} \to \mathcal{M}$, we denote its derivative by $\textup{d}\kappa(\hat{D}):\mathcal{T}_{\hat{D}}\mathcal{T}_{U}\mathcal{M}\to\mathcal{T}_{U}\mathcal{M}$\footnote[1]{For any $\hat{D}\in\mathcal{T}_{U}\mathcal{M}$, the linear space $\mathcal{T}_{\hat{D}}\mathcal{T}_U\mathcal{M}$ is isomorphic to $\mathcal{T}_U\mathcal{M}$. Hence, $\textup{d}\kappa$ can be viewed as a mapping within $\mathcal{T}_{U}\mathcal{M}$.}, which satisfies
\begin{equation}\label{differential}
\lim_{\|\delta D\|\to 0}\frac{\|\kappa(\hat{D}+\delta D) - \kappa(\hat{D}) - \textup{d}\kappa(\hat{D})[\delta D]\|}{\|\delta D\|}=0.
\end{equation}
The ``retraction" is then defined as follows \cite{AbMaSe}.

\begin{definition}
A retraction $\mathcal{R}: \ \mathcal{T}\mathcal{M} \to \mathcal{M}$ on a manifold $\mathcal{M}$ is a smooth mapping satisfying
\begin{eqnarray*}
\mathcal{R}_{U}({\bf 0}) &=& U, \\
\textup{d}\mathcal{R}_{U}({\bf 0}) &=& \textup{Id}_{\mathcal{T}_{U}\mathcal{M}},
\end{eqnarray*}
where $\mathcal{R}_{U}$ is the restriction of $\mathcal{R}$ to $\mathcal{T}_{U}\mathcal{M}$ when $U\in\mathcal{M}$, $\bf 0$ denotes the zero element in $\mathcal{T}_{U}\mathcal{M}$, and $\textup{Id}_{\mathcal{T}_{U}\mathcal{M}}$ is the identity mapping on $\mathcal{T}_{U}\mathcal{M}$.
\end{definition}

In our discussion, for simplicity, we introduce a macro $$\text{ortho}(U,D,t) = \mathcal{R}_{[U]}(tD),$$ for $U\in\mathcal{M}^N$ and $D\in\mathcal{T}_{[U]}\mathcal{G}^N$, which is a smooth curve on $\mathcal{M}^N$ starting from $U$ and along the initial direction $D$ when considered as an operator with respect to $t$. More precisely, the smooth mapping $\text{ortho}(U,D,t)$ satisfies that
\begin{eqnarray}
\text{ortho}(U,D,0)=U, \label{retraction1}\\
\frac{\partial}{\partial t}\text{ortho}(U,D,0)=D. \label{retraction2}
\end{eqnarray}
Moreover, if \eqref{retraction1} and \eqref{retraction2} hold true for all $U\in\mathcal{M}^N$ and $D\in\mathcal{T}_{[U]}\mathcal{G}^N$, then the corresponding $\mathcal{R}$ is indeed a retraction \cite{AbMaSe}.

Taking $\hat{D}={\bf 0}$ in \eqref{differential}, we have that
\begin{equation*}
\lim_{\|\delta D\|_{V^N}\to0}\frac{\|\textup{ortho}(U,\delta D,1)-\textup{ortho}(U,\delta D,0)-\frac{\partial}{\partial t}\text{ortho}(U,\delta D,0)\|_{V^N}}{\|\delta D\|_{V^N}}=0
\end{equation*}
for any retraction $\text{ortho}$.
By using \eqref{retraction1} and \eqref{retraction2} and rewriting $\delta D = tD$, we obtain
\begin{equation*}
\lim_{t\|D\|_{V^N}\to0}\frac{\|\textup{ortho}(U,D,t)-U-tD\|_{V^N}}{t\|D\|_{V^N}}=0,
\end{equation*}
which indicates that
\begin{equation}\label{orthoerror}
\|\textup{ortho}(U,D,t)-U-tD\|_{V^N} = o(t\|D\|_{V^N}).
\end{equation}

Now we state an abstract line search method for an orthogonality constrained problem.
\begin{algorithm}\label{linesearch}
\caption{{\bf Line search method with orthogonality constraints}}
  Given the tolerance $\epsilon \in (0,1)$, the initial guess $U_0, \ s.t. \   U_0^TU_0  = I_N$, compute $\nabla_G E(U_0)$, and set $n = 0$;

 \While {$\|\nabla_G E(U_n)\|_F> \epsilon$}{
  Determine $D_n\in\mathcal{T}_{[U_n]}\mathcal{G}^N$ by a certain strategy;

  Find a suitable $t_n$;

  Update $$U_{n+1}=\text{ortho}(U_n,D_n,t_n);$$

  set $n=n+1$ and compute $\nabla_G E(U_n)$;
 }
\end{algorithm}

In order to ensure the convergence of Algorithm \ref{linesearch}, we should impose some restrictions on search directions $\{D_n\}_{n=0}^\infty$ and step sizes $\{t_n\}_{n=0}^\infty$.

For the search directions, we always demand that all $D_n$ are descent directions so that we may expect some function value reduction at each iteration. More precisely, we require
\begin{equation}\label{decrease-direction}
\langle\nabla_G E(U_n), D_n\rangle_{V^N}<0,\   n=0,1,2,\dots
\end{equation}
Meanwhile, it is undesirable to see that the search directions are almost orthogonal to the gradient directions, i.e. coincident to the contours, since the objection function value is nearly invariant through these directions. As a result, it is reasonable to restrict that
\begin{equation}\label{not-ortho}
\limsup_{n\to\infty}\frac{-\langle\nabla_G E(U_n), D_n\rangle_{V^N}}{\|\nabla_G E(U_n)\|_{V^N}^a}\ne 0
\end{equation}
for some $a>0$.

\begin{remark}
Such a direction $D_n$ is always attainable as long as $\nabla_G E(U_n)\ne0$. If there exists some $n$, of which $D_n$ does not satisfy \eqref{decrease-direction} or \eqref{not-ortho} when carry out Algorithm \ref{linesearch}, then one can reset $D_n = -\nabla_G E(U_n)$ to satisfy \eqref{not-ortho} for $a = 2$ and continue.
\end{remark}

To choose a suitable step size, we define
$$\phi_n(t) = E(\text{ortho}(U_n,D_n,t)), t\ge0,$$
and see that there exists a global minimizer $t_n^\ast$ such that $$\phi_n(t)\ge\phi_n(t_n^\ast), \forall t\ge0$$ provided $\phi_n(t)$ is bounded below. Theoretically, $t_n^\ast$ would be the optimal choice for the step size. However, it usually costs too much or even impossible to get the exact $t_n^\ast$.  Therefore, some inexact line search conditions are investigated.

One of the most famous conditions imposed to the step sizes is the following Armijo condition which has been studied and applied in a number of works (see, e.g., \cite{AbMaSe,DLZZ,HMWY} and references cited therein). By the Armijo condition, the step size $t_n$ is chosen to satisfy
\begin{equation}\label{Armijo}
E(\textup{ortho}(U_n,D_n,t_n))-E(U_n)\leq\eta t_n\langle\nabla_G E(U_n), D_n\rangle_{V^N},
\end{equation}
where $\eta\in(0,1)$ is a given parameter. We see from $\langle\nabla_G E(U_n), D_n\rangle_{V^N}<0$ that the objective function decreases monotonely during the iterations. In this case, a line search method is said to be a monotone line search method.

The monotone condition \eqref{Armijo} seems too strict in some cases. Instead, Zhang et al. \cite{ZH} introduced the following non-monotone condition that the step sizes $t_n$ satisfy
\begin{align}\label{back-cond1}
  E(\textup{ortho}(U_n,D_n,t_n))-\mathcal{C}_n \leq   \eta t_n\langle \nabla_GE(U_n),D_n\rangle_{V^N},  n=0, 1, 2, \dots
\end{align}
Here,
\begin{equation}\label{Cn}
\begin{cases}
\mathcal{C}_0=E(U_0), Q_0=1,  \\
Q_n=\alpha Q_{n-1}+1, \\
\mathcal{C}_n=\big(\alpha Q_{n-1}\mathcal{C}_{n-1}+E(U_{n})\big)/Q_n,
\end{cases}
\end{equation}
with $\alpha\in[0,1)$, a given parameter.
\begin{remark}
Note that
\begin{equation}\label{non-mono}
\mathcal{C}_n-E(U_n) = \frac{Q_n-1}{Q_n}(\mathcal{C}_{n-1}-E(U_n))\ge0,
\end{equation}
the value of the objective function does not necessarily decrease, which is the reason why \eqref{back-cond1} is called a ``non-monotone" condition.
\end{remark}

Since Armijo condition \eqref{Armijo} is simply a special case of \eqref{back-cond1} by taking $\alpha = 0$, we always consider \eqref{back-cond1} in the rest of this paper. We observe that for $t_n$ small enough, \eqref{back-cond1} will be always satisfied since
$$E(\textup{ortho}(U_n,D_n,t_n))-\mathcal{C}_n \leq E(\textup{ortho}(U_n,D_n,t_n))-E(U_n) \leq\eta t_n\langle \nabla_GE(U_n),D_n\rangle_{V^N}$$
for sufficient small $t_n$. To avoid an extreme small step size, which may cause the slow convergence of the algorithm, we require
\begin{equation}\label{boundbelow}
\liminf_{n\to\infty}t_n>0.
\end{equation}

The following theorem shows that Algorithm \ref{linesearch} with such search directions and step sizes terminates in finite steps and returns a stationary point.
\begin{theorem}\label{l1converge}
Suppose the sequence  $\{U_n\}_{n=0}^\infty$ is generated by
Algorithm \ref{linesearch}. If $\{D_n\}_{n=0}^\infty, \{t_n\}_{n=0}^\infty$ are chosen to satisfy \eqref{decrease-direction}, \eqref{not-ortho} and \eqref{back-cond1}, \eqref{boundbelow}
respectively, then either $\|\nabla_G E(U_n)\|_{V^N}=0$ for some positive integer $n$ or
\begin{equation}\label{conc-der}
\liminf_{n\to\infty} \| \nabla_G E(U_n)\|_{V^N}  = 0.
\end{equation}
\end{theorem}
\begin{proof}
Suppose $\|\nabla_G E(U_n)\|_{V^N}=0$ for some positive integer $n$, the conclusion is trivial. Assume that $$\|\nabla_G E(U_n)\|_{V^N}\neq0,  n=0,1,2,\dots$$

We obtain by the definition of $\mathcal{C}_n$ that
\begin{eqnarray*}
E(U_{n+1})-E(U_n) &=& E(U_{n+1})-\mathcal{C}_n+\mathcal{C}_n-E(U_n) \\
&=& E(U_{n+1})-\mathcal{C}_n - \frac{\alpha Q_{n-1}}{Q_n}(E(U_n)-\mathcal{C}_{n-1}), \ n=1,2,\dots
\end{eqnarray*}
and
\begin{equation*}
E(U_1)-E(U_0) = E(U_1)-\mathcal{C}_0.
\end{equation*}
Summing up all $n\in\mathbb{N}$ gives that
\begin{eqnarray*}
\lim_{i\to\infty}\sum_{n=0}^{i}\big(E(U_n)-E(U_{n+1})\big) &=& E(U_0)-\lim_{i\to\infty}E(U_i) \\
&=&-\sum_{n=0}^{\infty}\frac{1}{Q_{n+1}}(E(U_{n+1})-\mathcal{C}_n)\\
&\geq&-\eta\sum_{n=0}^{\infty}\frac{1}{Q_{n+1}}t_n\langle\nabla_G E(U_n), D_n\rangle_{V^N},
\end{eqnarray*}
where \eqref{back-cond1} is used in the last inequality.

Since $\displaystyle Q_n = \sum_{i=0}^n \alpha^i$, we have $Q_n\in[0, \frac{1}{1-\alpha})$. Then, we have from \eqref{decrease-direction} that
$$\sum_{n=0}^{\infty}-t_n\langle\nabla_G E(U_n), D_n\rangle_{V^N}<+\infty.$$
Thus,
$$\lim_{n\to\infty}-t_n\langle\nabla_G E(U_n), D_n\rangle_{V^N}=0,$$
or equivalently,
$$\lim_{n\to\infty}\frac{-\langle\nabla_G E(U_n), D_n\rangle_{V^N}}{\|\nabla_G E(U_n)\|_{V^N}^a}t_n\|\nabla_G E(U_n)\|_{V^N}^a=0.$$
By \eqref{not-ortho}, we have
$$\liminf_{n\to\infty} t_n\| \nabla_G E(U_n)\|_{V^N}^a  = 0,$$
i.e., there exists an subsequence $\{U_{n_j}\}_{j=0}^\infty$, such that
$$\lim_{j\to\infty} t_{n_j}\| \nabla_G E(U_{n_j})\|_{V^N}^a  = 0.$$
Note that $$\liminf_{j\to\infty} t_{n_j}\ne0,$$
we obtain
$$\lim_{j\to\infty} \| \nabla_G E(U_{n_j})\|_{V^N}^a  = 0,$$ which implies that
$$\liminf_{n\to\infty} \| \nabla_G E(U_{n})\|_{V^N}  = 0.$$
\end{proof}

\begin{remark}
We claim that \eqref{boundbelow} is not necessarily required.
In fact, \eqref{not-ortho} indicates that there exists an subsequence $\{n_j\}_{j=0}^\infty$, such that
$$\lim_{j\to\infty}\frac{-\langle\nabla_G E(U_{n_j}), D_{n_j}\rangle_{V^N}}{\|\nabla_G E(U_{n_j})\|_{V^N}^a}\ne 0.$$
We see from the proof of Theorem \ref{l1converge} that  \eqref{boundbelow} can be replaced by: for the subsequence $\{n_j\}_{j=0}^\infty$, there holds
\begin{equation}\label{sumunbound}
\sum_{j=0}^{\infty}t_{n_j}=\infty.
\end{equation}
It is worth mentioning that \eqref{boundbelow} typically leads to \eqref{sumunbound} and \eqref{sumunbound} does not demand the step sizes $\{t_n\}_{n=0}^\infty$ to be bounded from below.
\end{remark}

Here, we review the classic  ``backtracking" approach to get the suitable step sizes that satisfy the mentioned conditions and analyze the convergence of the line search algorithm equipped with suitable search directions and the ``backtracked" step sizes.

The following algorithm is indeed the so called ``Armijo-type backtracking" method \cite{LA,ZH}:

\begin{algorithm}\label{backtracking}
\caption{{\bf Backtracked step size strategy}($U, D, t^{\textup{initial}}, t_{\textup{min}}, \eta, k, \mathcal {C})$}

  Set $t=\max{(t^{\text{initial}}, t_{\textup{min}})}$;

  \While {$E(\textup{ortho}(U,D,t))-\mathcal{C} >   \eta t\langle \nabla_GE(U),D\rangle_{V^N}$}{
    $t = kt$;
  }

    Return $t$.
\end{algorithm}

Here and hereafter, $U\in\mathcal{M}^N$ is a feasible iteration point, $D\in\mathcal{T}_{[U]}G^N$ denotes the search direction, $t^{\textup{initial}}$ is the initial guess of the step size, $t_{\textup{min}}$ is an extreme small positive constant to prevent the step size to be zero in programming and $\eta,\ k$ are some given parameters. We understand that the choice of initial step size is strongly related to the search direction. For instance, the possible initial guess can be chosen as the Barzilai-Borwein(BB) step size for gradient methods \cite{DZ,DYH,WY,ZZWZ}, the so called ``Hessian based step size" for CG method \cite{DLZZ} and the constant step size one for Newton methods \cite{DZZ,HMWY,Smith94,ZBJ}. A line search method equipped with the backtracked step size strategy Algorithm \ref{backtracking} reads as follows:

\begin{algorithm}\label{linesearch2}
\caption{{\bf Backtracking-based line search method}}
 Given $\epsilon, \eta, \alpha, k\in(0,1), \ t_{\textup{min}}>0,$ the initial value $U_0, \ s.t. \   U_0^TU_0  = I_N$, compute $\nabla_G E(U_0)$, and set $n = 0$;

 \While {$\|\nabla_G E(U_n)\|_F> \epsilon$}{

  Compute $\mathcal{C}_n$ by \eqref{Cn};

  Determine $D_n\in\mathcal{T}_{[U_n]}\mathcal{G}^N$ by a certain strategy;

  Given the initial guess of the step size $t_n^{\text{initial}}$ by a certain strategy;

  Compute $$t_n =  \textup{{\bf Backtracked step size strategy}}(U_n, D_n,t_n^{\text{initial}}, t_{\textup{min}}, \eta, k, \mathcal{C}_n);$$

  Update $$U_{n+1}=\text{ortho}(U_n,D_n,t_n);$$

  Set $n=n+1$ and compute $\nabla_G E(U_n)$;
 }
\end{algorithm}

We need to impose the following assumption on the search directions $\{D_n\}_{n=0}^\infty$ to establish the convergence result.

\begin{assume}\label{Dbound}
For the subsequence $\{n_j\}_{j=0}^\infty$ that satisfies
$$\lim_{j\to\infty}-\frac{\langle \nabla_G E(U_{n_j}), D_{n_j}\rangle_{V^N}}{\|\nabla_G E(U_{n_j})\|_{V^N}^a}=\delta>0,$$
there exists a constant $C>0$ such that
\begin{equation}
\|D_{n_j}\|_{V^N}\le C, \  j=0,1,2,\dots
\end{equation}
\end{assume}

We now show that such $\{t_n\}_{n=0}^\infty$ obtained by Algorithm \ref{backtracking} leads to a line search method that is convergent.

\begin{theorem}\label{btbb}
Suppose the sequence $\{U_n\}_{n=0}^\infty$ is generated by Algorithm \ref{linesearch2}, $\{D_n\}_{n=0}^\infty$ is the set of corresponding search directions satisfying \eqref{decrease-direction}, \eqref{not-ortho} and Assumption \ref{Dbound}, then either $\|\nabla_G E(U_n)\|_{V^N}=0$ for some positive integer $n$ or
\begin{equation*}
\liminf_{n\to\infty} \| \nabla_G E(U_n)\|_{V^N}  = 0.
\end{equation*}
\end{theorem}

By applying Theorem \ref{btbb} to some existing orthogonality constrained line search methods, we can loosen the convergence conditions therein. More details and the proof of Theorem \ref{btbb} are referred to Appendix \ref{A1}.

We see from Theorem \ref{btbb} that Algorithm \ref{backtracking} can generate a sequence of step sizes $\{t_n\}_{n=0}^\infty$, which together with suitable search directions $\{D_n\}_{n=0}^\infty$ leads to a converged line search method. However, we need to carry out $$U_{n+1}(t)=\textup{ortho}(U_n,D_n,t)$$ and the corresponding objective function value $E(U_{n+1}(t))$ once a backtracking step in Algorithm \ref{backtracking}, which occupy the main part of computations at each backtracking step. 
It can be predicted that the total cost at an iteration is strongly depends on the times of backtracking since the cost of each backtracking step is nearly the same (c.f. Fig. \ref{f0-1} and Fig. \ref{f0-2}). 
To get rid of this drawback, some new step size strategies which avoid computing $U_{n+1}(t)$ explicitly are of interest.

\section{Adaptive step size strategy}\label{sec-asss}
In this section, we propose and analyze
an adaptive step size strategy for orthogonality constrained line search methods. We will see that our adaptive step size strategy can provide better step sizes $\{t_n\}_{n=0}^\infty$ more efficiently than the Armijo-type backtracking approach.

We should introduce some notation before we propose our adaptive strategy.
Suppose $E(U)$ is of second order differentiable. We denote the second order derivative of $E(U)$ by $\nabla^2 E(U)$. Then we get from \cite{EAS} that
the Hessian of $E(U)$ on the Grassmann manifold is
\begin{equation*}
\nabla_G^2E(U)[D] =
(I-UU^T)\nabla^2E(U)[D] - DU^T\nabla E(U), \forall \ D \in
\mathcal{T}_{[U]}\mathcal{G}^N.
\end{equation*}
We sometimes denote
\begin{eqnarray*}
\langle\nabla_G^2E(U)[D_1], D_2\rangle_{V^N} = \textup{tr}(D_2^T\nabla^2E(U)[D_1]) - \textup{tr}(D_2^TD_1\nabla E(U)),
\end{eqnarray*}
by $\nabla_G^2E(U)[D_1,D_2]$ when $\ D_1,D_2 \in \mathcal{T}_{[U]}\mathcal{G}^N$ \cite{EAS}.

Let $[U], [W] \in \mathcal{G}^N$, with
$U,W\in \mathcal{M}^N$. We obtain from Lemma A.1 in \cite{DLZZ} that there exists a geodesic
\begin{equation}\label{geodesic1}
 \Gamma(t)=[UA\cos{(\Theta t)}A^T+A_2\sin{(\Theta t)}A^T], t\in[0,1],
\end{equation}
such that
\begin{eqnarray*}
\Gamma(0)=[U], \Gamma(1)=[W].
\end{eqnarray*}
Here, $U^TW=A\cos{\Theta}B^T$ and $W-U(U^TW)=A_2\sin{\Theta}B^T$ is the singular value decomposition(SVD) of $U^TW$ and $W-U(U^TW)$ respectively, $$\Theta=\textup{diag}(\theta_1, \theta_2, \dots, \theta_N)$$ is a diagonal matrix with $\theta_i\in[0, \pi/2]$ and $$\sin{(\Theta t)}=\diag(\sin(\theta_1 t), \sin(\theta_2 t), \dots, \sin(\theta_N t))$$ with similar notation for $\cos{(\Theta t)}$. Without loss of generality, we may assume here and hereafter that $\theta_1\ge\theta_2\ge\cdots\ge\theta_n$. Note that $A_2\in \mathcal{M}^N$.

\begin{remark}\label{uni-geo}
For any $U\in\mathcal{M}^N, D\in\mathcal{T}_{[U]}\mathcal{G}^N$, let $D=ASB^T$ be the SVD of $D$ where $A\in\mathcal{T}_{[U]}\mathcal{G}^N$, $S,B\in\mathbb{R}^{N\times N}$, then there exists an unique geodesic
\begin{equation}\label{geodesic3}
 \Gamma(t)=[UB\cos{(S t)}B^T+A\sin{(S t)}B^T],
\end{equation}
which start from $[U]$ and along direction $D$ \cite{EAS}. The above expression \eqref{geodesic1} is just a special case with initial direction $A_2\Theta A^T$.
\end{remark}

More specifically, we use macro $[\textup{exp}_{[U]}(t D)]$
to denote the geodesic on $\mathcal{G}^N$ which starting from $[U]$ and with the initial direction $D\in\mathcal{T}_{[U]}\mathcal{G}^N$. It is easy to check that such geodesic is one of the retractions.
We now define the parallel mapping which maps a tangent vector along the geodesic \cite{EAS}.
\begin{definition}
The parallel mapping $\tau_{\scriptscriptstyle{(U,D,t)}}: \ \mathcal{T}_{[U]}\mathcal{G}^N\to\mathcal{T}_{[\textup{exp}_{[U]}(tD)]}\mathcal{G}^N$ along geodesic $[\textup{exp}_{[U]}(tD)]$ is defined as
\begin{equation*}
\tau_{\scriptscriptstyle{(U,D,t)}}\tilde{D}=\big((-UB\sin{(St)}A^T+A\cos{(St)}A^T+(I_N-AA^T)\big)\tilde{D},
\end{equation*}
where $D=ASB^T$ is the SVD of $D$.
\end{definition}

It can be verified that 
\begin{eqnarray}\label{par_invariant}
\|\tau_{\scriptscriptstyle{(U,D,t)}}\tilde{D}\|_F=\|\tilde{D}\|_F, \forall \tilde{D}\in\mathcal{T}_{[U]}\mathcal{G}^N.
\end{eqnarray}

To show the theory, we introduce two distances on Grassmann manifold $\mathcal{G}^N$: 
\begin{equation}\label{dist}
\begin{split}
&\textup{dist}_{cF}([U], [W]) = \min_{P\in \mathcal{O}^{N\times N}} \| U - WP \|_{V^N}, \\
&\textup{dist}_{geo}([U], [W]) =  \| A_2\Theta A^T \|_{V^N}.
\end{split}
\end{equation}

\begin{remark}
Denote $\|\cdot\|_F$ the Frobenius norm of matrix. It can be calculated that \cite{EAS}
\begin{equation*}
\begin{split}
&\textup{dist}_{cF}([U], [W]) = \|2\sin{\frac{\Theta}{2}}\|_F, \\
&\textup{dist}_{geo}([U], [W]) = \| \Theta \|_F,
\end{split}
\end{equation*}
which indicate that these two kinds of distance are equivalent, namely,
$$\textup{dist}_{cF}([U], [W])\leq\textup{dist}_{geo}([U], [W])\leq2\textup{dist}_{cF}([U], [W]).$$
In addition, we see that
\begin{equation}\label{DeTheta}
\|D\|_{V^N} = \|A_2\Theta A^T\|_{V^N} = \|\Theta\|_F = \textup{dist}_{geo}([U], [W]),
\end{equation}
where D is the initial direction of the geodesic \eqref{geodesic1}.
\end{remark}

To present our adaptive step size strategy and carry out the convergence proof, we need the following conclusion, which can be obtained from Remark 3.2 and Remark 4.2 of \cite{Smith94}.
\begin{proposition}\label{Taylor}
If $E(U)$ is of second order differentiable, then for all $U\in\mathcal{M}^N$, $D\in\mathcal{T}_{[U]}\mathcal{G}^N$, there exists an $\xi\in(0,t)$ such that
\begin{eqnarray}
E(\textup{exp}_{[U]}(tD)) &=& E(U)+t\langle\nabla_GE(\textup{exp}_{[U]}(\xi D)),\tau_{\scriptscriptstyle{(U,D,\xi)}} D\rangle_{V^N} \label{taylor1} \\
&=& E(U)+t\langle\nabla_GE(U),D\rangle_{V^N} \notag\\
&&+ \frac{t^2}{2}\nabla_G^2E(U)[D, D] + o(t^2\|D\|_{V^N}^2) \notag
\end{eqnarray}
and
\begin{eqnarray}
\label{taylor3}
\tau_{\scriptscriptstyle{(U,D,t)}}^{-1}\nabla_GE(\textup{exp}_{[U]}(tD)) &=& \nabla_G E(U)+t\tau_{\scriptscriptstyle{(U,D,\xi)}}^{-1}\nabla_G^2E(\textup{exp}_{[U]}(\xi D))[\tau_{\scriptscriptstyle{(U,D,\xi)}} D].
\end{eqnarray}
\end{proposition}

We are now able to introduce our adaptive step size strategy. Inspired by the well-known process of adaptive finite element method \cite{CKNS,CDGHZ,DHZ,DXZ}, our adaptive step size can be divided into the following steps:
\begin{center}
Initialize $\to$ {\bf Estimate $\to$ Judge $\to$ Improve}.
\end{center}
We suppose that the initial guess of the step size at the $n$-th iteration $t_n^{\textup{initial}}$ is given.

{\bf Estimate.}
 As mentioned in Section \ref{sec-pre}, the final step size $t_n$ is supposed to satisfy \eqref{Armijo} or \eqref{back-cond1}. However, predicting $E(U_{n+1}(t_n))$ in \eqref{Armijo} or \eqref{back-cond1} need to compute the trail point $U_{n+1}(t_n) = \textup{ortho}(U_n,D_n,t_n)$ and the corresponding functional value, which are usually expensive. Instead, we consider the objective function $E$ around $U_n$ as follows
\begin{equation}\label{approx}
E(U_{n+1}(t)) \approx\  E(U_n) + t \langle \nabla_G E(U_n),D_n\rangle_{V^N} + \frac{t^2}{2}\nabla_G^2E(U_n)[D_n, D_n].
\end{equation}
Replacing the term $E(\textup{ortho}(U_n, D_n, t))$ in \eqref{back-cond1} by the right hand side of \eqref{approx}, we have
\begin{align*}
  E(U_n)+　t \langle \nabla_G E(U_n),D_n\rangle_{V^N} + \frac{t^2}{2}\nabla_G^2E(U_n)[D_n, D_n]-\mathcal{C}_n \leq   \eta t\langle \nabla_GE(U_n),D_n\rangle_{V^N},
\end{align*}
or equivalently,
\begin{align*}
  \frac{E(U_n)+　t \langle \nabla_G E(U_n),D_n\rangle_{V^N} + \frac{t^2}{2}\nabla_G^2E(U_n)[D_n, D_n]-\mathcal{C}_n}{t\langle \nabla_GE(U_n),D_n\rangle_{V^N}} \geq   \eta .
\end{align*}

Hence, we propose the following estimator:
\begin{equation}\label{indicator4}
\zeta_n(t) = \frac{E(U_n)-\mathcal{C}_n+t\langle \nabla_G E(U_n),D_n\rangle_{V^N} + \frac{t^2}{2}\nabla_G^2 E(U_n)[D_n, D_n]}{t\langle \nabla_G E(U_n),D_n\rangle_{V^N}}
\end{equation}
to guide us whether to accept a step size or not at iteration $n$.

To use the estimator \eqref{indicator4}, it is reasonable to restrict $t_n\|D_n\|_{V^N}\leq\vartheta_n$ for some small $\vartheta_n$ since \eqref{approx} remains reliable only in a neighborhood of $U_{n}$. We first set $$t_n = \min{(t_n^{\textup{initial}}, \frac{\vartheta_n}{\|D_n\|_{V^N}})}$$ and then calculate the estimator $\zeta_n(t_n)$.

{\bf Judge.}
The step size $t_n$ is said to be acceptable if $$\zeta_n(t_n)\geq\eta,$$ where $\eta\in(0,1)$ is some given parameter.
Otherwise, $t_n$ is to be improved.

We see from a simple calculation that $t_n>0$ is acceptable if and only if
\begin{equation}\label{stepsizeinterval}
t_n\leq
\begin{cases}
\min{(\frac{[(\eta-1)-\sqrt{\Delta}]\langle \nabla_G E(U_n),D_n\rangle_{V^N}}{\nabla_G^2E(U_n)[D_n, D_n]},\frac{\vartheta_n}{\|D_n\|_{V^N}})}, & \mbox{if $\nabla_G^2E(U_n)[D_n, D_n]>0$,} \\
\frac{\vartheta_n}{\|D_n\|_{V^N}}, & \mbox{otherwise},
\end{cases}
\end{equation}
where $$\Delta = (\eta-1)^2-\frac{2\nabla_G^2E(U_n)[D_n, D_n](E(U_n)-\mathcal{C}_n)}{\langle \nabla_G E(U_n),D_n\rangle_{V^N}^2}.$$

{\bf Improve.}
 If $t_n$ is not acceptable, we choose the step size $t_n$ to be the minimizer of $$E(U_n)+t\langle \nabla_G E(U_n),D_n\rangle_{V^N}+\frac{t^2}{2}\nabla_G^2E(U_n)[D_n, D_n]$$ within the interval given by \eqref{stepsizeinterval}, that is,
\begin{equation}\label{stepsize}
t_n=
\begin{cases}
\min\left(-\frac{\langle \nabla_G E(U_n),D_n\rangle_{V^N}}{\nabla_G^2E
(U_n)[D_n, D_n]}, \frac{\vartheta_n}{\|D_n\|_{V^N}}\right), & \mbox{if
 $\nabla_G^2E(U_n)[D_n, D_n]>0$},\\
\frac{\vartheta_n}{\|D_n\|_{V^N}}, & \mbox{otherwise}.
\end{cases}
\end{equation}

Taking the whole procedure into account, we have our adaptive step size strategy as Algorithm \ref{adaptive}:
\begin{algorithm}\label{adaptive}
\caption{{\bf Adaptive step size strategy}($U, D, t^{\text{initial}}, t_{\textup{min}}, \eta, \theta, \mathcal{C}$)}
  Set $t=\min{(\max{(t^{\text{initial}}, t_{\textup{min}})}, \theta/\|D\|_{V^N})}$;

  Calculate estimator $$\zeta(t) = \frac{E(U)-\mathcal{C}+t\langle \nabla_G E(U),D\rangle_{V^N} + \frac{t^2}{2}\nabla_G^2 E(U)[D, D]}{t\langle \nabla_G E(U),D\rangle_{V^N}};$$

  \If {$\zeta(t)<\eta$}{

  Choose \begin{equation*}
t=
\begin{cases}
\min\left(-\frac{\langle \nabla_G E(U),D\rangle_{V^N}}{\nabla_G^2E
(U)[D, D]}, \frac{\theta}{\|D\|_{V^N}}\right), & \mbox{if
 $\nabla_G^2E(U)[D, D]>0$},\\
\frac{\theta}{\|D\|_{V^N}}, & \mbox{otherwise};
\end{cases}
\end{equation*}
}

Return $t$;
\end{algorithm}

The corresponding adaptive line search method can thus be written as the following Algorithm \ref{linesearch3}:
\begin{algorithm}\label{linesearch3}
\caption{{\bf Adaptive step size line search method}}
 Given $\epsilon, \eta, \alpha\in(0,1), \ t_{\textup{min}}>0,$ the initial value $U_0, \ s.t. \   U_0^TU_0  = I_N$, compute $\nabla_G E(U_0)$, and set $n = 0$;

 \While {$\|\nabla_G E(U_n)\|_F> \epsilon$}{
  Choose a suitable $\vartheta_n$;

  Compute $\mathcal{C}_n$ by \eqref{Cn};

  Determine $D_n\in\mathcal{T}_{[U_n]}\mathcal{G}^N$ by a certain strategy;

  Given the initial guess of the step size $t_n^{\text{initial}}$ by a certain strategy;

  Compute $$t_n = \textup{{\bf Adaptive step size strategy}}(U_n, D_n,t_n^{\text{initial}}, t_{\textup{min}}, \eta, \vartheta_n, \mathcal{C}_n);$$

  Update $$U_{n+1}=\text{ortho}(U_n,D_n,t_n);$$

  Set $n=n+1$ and compute $\nabla_G E(U_n)$;
 }
\end{algorithm}

We see from Algorithms \ref{adaptive} and \ref{linesearch3} that our step size strategy requires the information about (Grassmann) Hessian $\nabla^2_G E(U_n)[D_n,D_n]$ in {\bf Estimate} step at each iteration.
However, when compared with the backtracking approach, our strategy needs not to compute the trial point and the corresponding function value repeatedly, which is the most expensive part in orthogonality constrained line search methods. As a result, the total cost at each iteration may decrease. In addition, our adaptive strategy will give a reasonable step size which is either the initial guess recommended by some classic step size strategy or the minimizer of the second order approximation of the objective function around the current iteration point. For comparison, the backtracking procedure gives an acceptable but unassessable number that satisfies \eqref{back-cond1}. One can never say that it is a persuasive one among the set:
$$\{t\in\mathbb{R}^+:t \ \text{satisfies} \ \eqref{back-cond1}\}.$$

To establish the convergence theory of Algorithm \ref{linesearch3}, we need the following assumption:
\begin{assume}\label{assum_bound}
Grassmann Hessian $\nabla_G^2 E(U)$ is bounded, that is, there exist $\bar{C}>0$
such that
\begin{equation}\label{hess-bound}
\| \nabla_G^2 E(U)[D]\|_{V^N} \leq \bar{C}\|D\|_{V^N}, \ \ \forall \ U\in \mathcal{M}^N, D\in\mathcal{T}_{[U]}\mathcal{G}^N.
\end{equation}
\end{assume}

We see from \eqref{taylor3} that \eqref{hess-bound} typically results in
\begin{equation}\label{der-bound}
\| \nabla_G E(U)\|_{V^N} \leq C_0, \ \ \forall \ U\in \mathcal{M}^N,
\end{equation}
where $C_0$ can be chosen as $2N\bar{C}$. Also, \eqref{der-bound} holds simply owing to $\mathcal{M}^N$ is compact.

The following theorem shows the convergence of Algorithm \ref{linesearch3} if we choose $\{\vartheta_n\}_{n=0}^\infty$ properly:
\begin{theorem}\label{mainconvergence}
Suppose E(U) is of second order differentiable and let Assumption \ref{assum_bound} holds true. If $\{D_n\}_{n=0}^\infty$ is chosen to satisfy \eqref{decrease-direction}, \eqref{not-ortho} and Assumption \ref{Dbound}, then there exists a positive sequence $\{\vartheta_n\}_{n=0}^\infty$, such that for the sequence  $\{U_n\}_{n=0}^\infty$  generated by
Algorithm \ref{linesearch3},
either $\|\nabla_G E(U_n)\|_{V^N}=0$ for some positive integer $n$ or
\begin{equation*}
\liminf_{n\to\infty} \| \nabla_G E(U_n)\|_{V^N}  = 0.
\end{equation*}
\end{theorem}
\begin{proof}
By Theorem \ref{l1converge}, it is sufficient to prove that there exists a positive sequence $\{\vartheta_n\}_{n=0}^\infty$, such that the step size $t_n$ satisfies \eqref{back-cond1} and \eqref{sumunbound}.
We see from Algorithm \ref{linesearch3} that every $t_n$ is chosen to satisfy
\begin{eqnarray}
\zeta_n(t_n)&\ge&\eta, \\
t_n\|D_n\|&\le&\vartheta_n,
\end{eqnarray}
which imply that
$$E(U_n)+t\textup{tr}(\nabla_G E(U_n)^TD_n)+\frac{t^2}{2}\nabla_G^2 E(U_n)[D_n,D_n]-\mathcal{C}_n\le\eta t_n\textup{tr}(\nabla_G E(U_n)^TD_n).$$
Let
\begin{eqnarray*}
\vartheta_n=\sup\{\tilde{\vartheta}_n: &&E(\textup{ortho}(U_n,D_n,t))-E(U_n)-t\textup{tr}(\nabla_G E(U_n)^TD_n) \\
&&-\frac{t^2}{2}\nabla_G^2 E(U_n)[D_n,D_n]\le-\frac{\eta t\textup{tr}(\nabla_G E(U_n)^TD_n)}{2}, \forall t\le\frac{\tilde{\vartheta}_n}{\|D_n\|_{V^N}}\}.
\end{eqnarray*}
Then we obtain from the definition of $E(U_{n+1})$ and $\vartheta_n$ that
$$E(U_{n+1})-\mathcal{C}_n\le\frac{\eta}{2} t_n\textup{tr}(\nabla_G E(U_n)^TD_n), \forall n\in\mathbb{N}_0,$$
i.e., \eqref{back-cond1} holds.

As for \eqref{sumunbound}, we only need to take subsequence $\{n_j\}_{j=0}^{\infty}$ such that
$$\lim_{j\to\infty}-\frac{\textup{tr}(\nabla_G E(U_{n_j})^TD_{n_j})}{\|\nabla_G E(U_{n_j})\|_{V^N}^a}=\delta>0$$
into account.

The corresponding $t_{n_j}$ has only three options, say, $$t_{n_j}=\max{(t_{n_j}^{\text{initial}}, t_{\textup{min}})},$$ $$t_{n_j}=\frac{-\textup{tr}(\nabla_G E(U_{n_j})^TD_{n_j})}{\nabla_G^2 E(U_{n_j})[D_{n_j}, D_{n_j}]},$$ or $$t_{n_j}=\frac{\vartheta_{n_j}}{\|D_{n_j}\|_{V^N}}.$$ So there is at least one infinite subsequence of $\{n_j\}_{j=0}^{\infty}$, which is, with out loss of generality, also denoted by $\{n_j\}_{j=0}^{\infty}$, such that

{\bf Case 1.} $t_{n_j}=\max{(t_{n_j}^{\text{initial}}, t_{\textup{min}})}$. We have immediately $$\sum_{j=0}^{\infty} t_{n_j}\ge\sum_{j=0}^{\infty} t_{\textup{min}} =+\infty.$$

{\bf Case 2.} $t_{n_j} = \frac{-\textup{tr}(\nabla_G E(U_{n_j})^TD_{n_j})}{\nabla_G^2 E(U_{n_j})[D_{n_j}, D_{n_j}]}$. We obtain from Assumptions \ref{Dbound} and \ref{assum_bound} that
\begin{eqnarray*}
t_{n_j}&\ge&\frac{-\textup{tr}(\nabla_G E(U_{n_j})^TD_{n_j})}{\bar{C}\|D_{n_j}\|_{V^N}^2} \\
&\ge&\frac{-\textup{tr}(\nabla_G E(U_{n_j})^TD_{n_j})}{\|\nabla_G E(U_{n_j})\|_{V^N}^a}\frac{\|\nabla_G E(U_{n_j})\|_{V^N}^a}{\bar{C}\|D_{n_j}\|_{V^N}^2} \\
&\ge&\frac{\delta}{\bar{C}C^2}\|\nabla_G E(U_{n_j})\|_{V^N}^a.
\end{eqnarray*}
Hence, either $$\sum_{j=0}^{\infty} t_{n_j}=+\infty$$
or $$\lim_{j\to\infty}\|\nabla_G E(U_{n_j})\|_{V^N}=0.$$

{\bf Case 3.} $t_{n_j} = \frac{\vartheta_{n_j}}{\|D_{n_j}\|_{V^N}}$.
If
$$\liminf_{j\to\infty} t_{n_j}>0,$$
then \eqref{sumunbound} is satisfied and we complete the proof.

Assume otherwise, i.e., there exists a subsequence also denoted by $\{n_j\}$ such that
$\lim_{j\to\infty} t_{n_j}=0,$
or equivalently $\lim_{j\to\infty}\vartheta_{n_j}=0$ thanks to Assumption \ref{Dbound}.

For simplicity, we sometimes denote $U_{n+1}(t) = \textup{ortho}(U_n,D_n,t)$, then $U_{n+1}=U_{n+1}(t_n)$. We have that for all $n\in\mathbb{N}_0$, there hold
\begin{eqnarray*}
&&E(U_{n+1}(t))-E(U_n)-t\textup{tr}(\nabla_G E(U_n)^TD_n)-\frac{t^2}{2}\nabla_G^2 E(U_n)[D_n,D_n]  \\
&=& E(U_{n+1}(t)) - E(\textup{exp}_{[U_n]}(tD_n)) + E(\textup{exp}_{[U_n]}(tD_n)) - E(U_n) \\
&&-t\textup{tr}(\nabla_G E(U_n)^TD_n)-\frac{t^2}{2}\nabla_G^2 E(U_n)[D_n,D_n] \notag := T_n^{(1)}+T_n^{(2)}, \notag
\end{eqnarray*}
where $$T_n^{(1)} = E(U_{n+1}(t)) - E(\textup{exp}_{[U_n]}(tD_n))$$ and $$T_n^{(2)} = E(\textup{exp}_{[U_n]}(tD_n)) - E(U_n)-t\textup{tr}(\nabla_G E(U_n)^TD_n)-\frac{t^2}{2}\nabla_G^2 E(U_n)[D_n,D_n].$$
We see from Remark \ref{uni-geo} that there exists a geodesic $[\textup{exp}_{[U_{n+1}(t)]}(t\hat{D})]$ such that
\begin{eqnarray*}
\textup{exp}_{[U_{n+1}(t)]}({\bf 0}) = U_{n+1}(t), \ [\textup{exp}_{[U_{n+1}(t)]}(\hat{D})] = [\textup{exp}_{[U_n]}(tD_n)],
\end{eqnarray*}
and obtain by \eqref{taylor1} that
\begin{eqnarray*}
|T_n^{(1)}| &=& |E(\textup{exp}_{[U_{n+1}(t)]}(0\hat{D})) - E(\textup{exp}_{[U_{n+1}(t)]}(\hat{D}))| \\
&=& |\langle\nabla_G E(\textup{exp}_{[U_{n+1}(t)]}(\xi \hat{D})), \tau_{\scriptscriptstyle{(U_{n+1}(t),\hat{D},\xi)}} \hat{D}\rangle_{V^N}| \\
&\leq& \|\nabla_G E(\textup{exp}_{[U_{n+1}(t)]}(\xi \hat{D}))\|_{V^N}\|\tau_{\scriptscriptstyle{(U_{n+1}(t),\hat{D},\xi)}} \hat{D}\|_{V^N} \\
&\leq& C_0\|\hat{D}\|_{V^N}, 
\end{eqnarray*}
where \eqref{der-bound} and \eqref{par_invariant} are used in the last inequality. By \eqref{DeTheta} and \eqref{orthoerror}, we get
\begin{eqnarray*}
\|\hat{D}\|_{V^N} &=& \textup{dist}_{geo}([U_{n+1}(t)], [\textup{exp}_{[U_n]}(tD_n)]) \\ &\leq& 2\textup{dist}_{cF}([U_{n+1}(t)], [\textup{exp}_{[U_n]}(tD_n)]) \\
&\leq& 2\|U_{n+1}(t) - \textup{exp}_{[U_n]}(tD_n)\|_{V^N} \\
&\leq& 2\big(\|\textup{ortho}(U_n, D_n, t) - U_n - tD_n\|_{V^N} \\
&&+ \|\textup{exp}_{[U_n]}(tD_n) - U_n - tD_n\|_{V^N}\big) \\
&=& o(t\|D_n\|_{V^N}),
\end{eqnarray*}
which leads to
\begin{equation}\label{I1}
T_n^{(1)} = o(t\|D_n\|_{V^N}).
\end{equation}
As for $T_n^{(2)}$, \eqref{taylor1} gives that
\begin{equation}\label{I2}
T_n^{(2)} = o(t^2\|D_n\|_{V^N}^2).
\end{equation}
Combining \eqref{I1} and \eqref{I2}, we arrive at
\begin{eqnarray*}
&&E(U_{n+1}(t))-E(U_n)-t\textup{tr}(\nabla_G E(U_n)^TD_n)-\frac{t^2}{2}\nabla_G^2 E(U_n)[D_n,D_n]  \\
&=& T_n^{(1)}+T_n^{(2)} = o(t\|D_n\|_{V^N}), ~~ \forall n\in\mathbb{N}_0.
\end{eqnarray*}

Note that the definition of $\vartheta_{n_j}$ implies that for all $n_j$, there exist $$t_{n_j}^\ast\in(\frac{\vartheta_{n_j}}{\|D_{n_j}\|_{V^N}}, \frac{\vartheta_{n_j}+\frac{1}{n_j}}{\|D_{n_j}\|_{V^N}})$$ such that
\begin{eqnarray}\label{contra}
o(t_{n_j}^\ast\|D_{n_j}\|_{V^N})  &=& E(\textup{ortho}(U_{n_j},D_{n_j},t_{n_j}^\ast))-E(U_{n_j}) \notag\\
&&-t_{n_j}^\ast\textup{tr}(\nabla_G E(U_{n_j})^TD_{n_j})-\frac{{t_{n_j}^\ast}^2}{2}\nabla_G^2 E(U_{n_j})[D_{n_j},D_{n_j}] \notag\\
&>&-\frac{\eta t_{n_j}^\ast\textup{tr}(\nabla_G E(U_{n_j})^TD_{n_j})}{2} \notag\\
&=& \frac{\eta}{2}\frac{ -\textup{tr}(\nabla_G E(U_{n_j})^TD_{n_j})}{\|\nabla_G E(U_{n_j})\|_{V^N}^a}\frac{1}{\|D_{n_j}\|_{V^N}}t_{n_j}^\ast\|D_{n_j}\|_{V^N}\|\nabla_G E(U_{n_j})\|_{V^N}^a.
\end{eqnarray}
It is easy to see that
$$0\le\lim_{j\to\infty}t_{n_j}^\ast\|D_{n_j}\|_{V^N}\le\lim_{j\to\infty}\big(\vartheta_{n_j}+\frac{1}{n_j}\big)=0.$$
Hence, by letting $j\to\infty$ in \eqref{contra}, we arrive at
$$0\ge\lim_{j\to\infty}\frac{\eta\delta}{2C_0}\|\nabla_G E(U_{n_j})\|_{V^N}^a,$$
which completes our proof.
\end{proof}

The above discussions indicate that a line search method equipped with some standard search directions and our adaptive step sizes globally converges to a stationary point under some mild assumptions. In addition, our step size strategy is much cheaper than Algorithm \ref{backtracking} at an iteration of a line search method where backtracking step occurs.
\section{Applications to electronic structure calculations}\label{sec-num}
In this section, we apply the adaptive step size strategy to a gradient type method to solve the Kohn-Sham energy minimization problem. We choose the negative gradient directions to be the search directions and the BB step sizes for the initial guesses $t_n^{\text{initial}}$ \cite{ZZWZ}. We then compare some different step size strategies to show the advantages of ours. 
\subsection{Kohn-Sham DFT model}
In Kohn-Sham DFT model, $U = (u_1, \dots, u_N)\in {\big(H^1(\mathbb{R}^3)\big)}^N$ and the objective functional $E_{\textup{KS}}(U)$ reads as
\begin{eqnarray}\label{energy}
E_{\textup{KS}}(U)&=&\frac{1}{2} \int_{\mathbb{R}^3} \sum_{i=1}^N|\nabla u_i(r)|^2 dr +
\frac{1}{2}\int_{\mathbb{R}^3}\int_{\mathbb{R}^3}\frac{\rho(r)\rho(r')}{|r-r'|}drdr' \nonumber\\
&&+\int_{\mathbb{R}^3} V_{ext}(r)\rho(r)dr+ \int_{\mathbb{R}^3}
\varepsilon_{xc}(\rho)(r)\rho(r)dr,
\end{eqnarray}
where $N$ denotes the number of electrons, $u_i$ are sometimes called the Kohn-Sham orbitals,
$\rho(r)=\sum \limits_{i=1}^N|u_i(r)|^2$ is the electronic density,
$V_{ext}(r)$ is the external potential generated by the nuclei, and $\varepsilon_{xc}(\rho)(r)$ is the exchange-correlation functional
which is not known explicitly. In practise, some approximation such as local density approximation (LDA),
generalized gradient approximation (GGA) or some other approximations has to be used \cite{Ma}.

As mentioned in Section \ref{sec-pre}, the Grassmann gradient of $E_{\textup{KS}}(U)$ is
\begin{equation*}
    \nabla_G E_{\textup{KS}}(U)  = (I - UU^T)\nabla E_{\textup{KS}}(U).
\end{equation*}
We see from \cite{DLZZ} that $\nabla E_{\textup{KS}}(U) = \mathcal{H}(\rho)U$, $$\mathcal{H}(\rho) = -\frac{1}{2}\Delta + V_{ext} + \int_{\mathbb{R}^3}\frac{\rho(r')}{|r-r'|}dr' +
v_{xc}(\rho),$$
and $$v_{xc}(\rho)
= \frac{\delta(\rho\varepsilon_{xc}(\rho))}{\delta\rho}.$$

In further, the Hessian of $E_{\textup{KS}}(U)$ on the Grassmann manifold $\mathcal{G}_N$ has the form \cite{DLZZ}
\begin{equation*}
\begin{split}
\nabla^2_GE(U)[D_1, D_2] &= \text{tr} ( {D_1}^T\mathcal{H}(\rho)
D_2) - \text{tr}( {D_1}^TD_2U^T\mathcal{H}(\rho)U) \\
&+ 2\int_{\mathbb{R}^3}\int_{\mathbb{R}^3}\frac{(\sum_iu_i(r)d_{1,i}(r))(\sum_ju_j(r')d_{2,j}(r'))}
{|r-r'|}drdr' \\
&+ 2\int_{\mathbb{R}^3}\frac{\delta^2(\varepsilon_{xc}(\rho)\rho)}{\delta\rho^2}(r)
(\sum_iu_i(r)d_{1,i}(r))(\sum_ju_j(r)d_{2,j}(r))dr
\end{split}
\end{equation*}
provided that the total energy functional is of second order differentiable, or more specifically,
the approximated exchange-correlation functional is of second order differentiable.
Here, $D_i = (d_{i,1}, d_{i,2}, \cdots, d_{i,N})(i = 1,2)$ belong to $\mathcal{T}_{[U]}\mathcal{G}^N$.

We may discrete Kohn-Sham model \eqref{energy} by the plane wave method, the local basis set method, or
some real space methods. In this paper, we focus on the real space method. If we choose the $N_g$-dimension space \\ $V_{N_g}\subset H^1(\mathbb{R}^3)$
to approximate $H^1(\mathbb{R}^3)$, then the associated discretized Kohn-Sham model
can be formulated as
\begin{equation}\label{E_dKS}
\min_{[U]\in \mathcal{G}^N_{N_g}} \ \ \ E_{KS}(U),
\end{equation}
where $\mathcal{G}^N_{N_g}$ is the discretized Grassmann manifold defined by
\begin{equation*}
\mathcal{G}^N_{N_g} = \mathcal{M}^N_{N_g}/\sim,
\end{equation*}
\begin{equation*}
\mathcal{M}^N_{N_g} = \{U \in(V_{N_g})^N:   U^TU  = I_N\}
\end{equation*}
is the discretized Stiefel manifold and the equivalent relation $\sim$ has the similar meaning to what we have mentioned in Section \ref{sec-pre}. Typically, $N\ll N_g$.

We refer to \cite{DLZZ} for more detailed expressions of the Kohn-Sham DFT model under other type of discretizations, for instance, the finite difference discretization.

\subsection{Numerical experiments}
One class of the most basic algorithms for orthogonality constrained problems are the gradient type methods, which have been investigated in \cite{AbMaSe,SRNB,Smith94,WY,ZZWZ}.
In \cite{ZZWZ}, the well known BB step size is applied to accelerate the gradient type algorithms. More precisely, the initial step size at iteration $n$ is chosen as
\begin{equation}\label{tau1}
\tilde{t}_{n,1}=\frac{tr({S_{n-1}}^TS_{n-1})}{|tr({S_{n-1}}^TY_{n-1})|}
\end{equation}
or
\begin{equation}\label{tau2}
\tilde{t}_{n,2}=\frac{|tr({S_{n-1}}^TY_{n-1})|}{tr({Y_{n-1}}^TY_{n-1})}.
\end{equation}
 Here,
 \begin{center}
 $S_{n-1}=U_n-U_{n-1}$, $Y_{n-1}=\nabla_G E(U_n)-\nabla_G E(U_{n-1}).$
 \end{center}
 The non-monotone backtracking procedure is then applied to guarantee the convergence. We point out that the algorithm proposed in \cite{ZZWZ} can be viewed as a special case of Algorithm \ref{linesearch} by choosing $D_n=-\nabla_G E(U_n)$ and $t_n$ by Algorithm \ref{backtracking} with initial step sizes \eqref{tau1} or \eqref{tau2}. We test the case that $t_n^{\textup{initial}} = \tilde{t}_{n,1}$, $t_n^{\textup{initial}} = \tilde{t}_{n,2}$ and 
 \begin{equation}\label{initialstep}
 t_n^{\textup{initial}} = 
 \begin{cases}
  \tilde{t}_{n,1},& \mbox{for $n$ odd},  \\
  \tilde{t}_{n,2},& \mbox{for $n$ even},
 \end{cases}
 \end{equation}
 respectively, and choose the overall better one, i.e., \eqref{initialstep}, in our numerical experiments. We refer to Appendix \ref{A2} for detailed comparisons. 

We apply the gradient method with different step size strategies on the software package
Octopus\footnote[1]{Octopus:www.tddft.org/programs/octopus.} (version 4.0.1), and carry out all numerical experiments on LSSC-IV in the State Key Laboratory of
Scientific and Engineering Computing of the Chinese Academy of Sciences. We choose LDA to approximate $v_{xc}(\rho)$ \cite{PeZu} and use the Troullier-Martins
norm conserving pseudopotential \cite{TrMa}. 

Our examples include several typical molecular systems: benzene ($C_6H_6$), aspirin ($C_9H_8O_4)$,
fullerene ($C_{60}$), alanine chain $(C_{33}H_{11}O_{11}N_{11})$, carbon
nano-tube ($C_{120}$), carbon clusters $C_{1015}H_{460}$ and $C_{1419}H_{556}$.
We use QR strategy as retraction, that is,
$$\textup{ortho}(U,D,t)=(U+t D)R^{-1},$$
where $R$ is a upper-triangular matrix such that $$R^TR = (U+t D)^T(U+t D) = I_N + t^2D^TD.$$

We show the detailed results obtained by the gradient method with different step size strategies in TABLE \ref{t1},
in which ``iter" means the number of iterations required to terminate the algorithm, $\| \nabla_G E\|_F$ forms the norm of the gradient when the algorithm terminates, ``W.C.T" is the total wall clock time spent to converge, and ``A.T.P.I" is the average wall clock time needed per iteration.

In TABLE \ref{t1}, GM-QR-Back means the non-monotone backtracking-based algorithm ($\alpha = 0.85$) proposed and applied in \cite{ZZWZ}. We use our estimator \eqref{indicator4} to generate our adaptive algorithm with $\alpha=0.85$, and the corresponding results are named as GM-QR-Adap. We should mention that $\alpha$ is chosen to be $0.85$ since it is recommended in \cite{ZH}. We choose the parameter $\vartheta_n = 0.2$, for all $n$. Among all our experiments, $\eta=1$e$-4$, which is recommended in \cite{NW}, $t_{\textup{min}}=1$e$-20$ and $k=0.5$. It is worth mentioning that the parameter $k$ is only used in backtracking-based algorithm and we understand that the performance of the backtracking-based method is $k$-dependent. A small $k$ may make some final step sizes $t_n$ too small and a big $k$ typically leads to a large amount of extra backtrackings. Hence, we choose $k=0.5$ to find a balance.
For all the systems except $C_{1015}H_{460}$ and $C_{1419}H_{556}$, $\epsilon$ is chosen to be $1$e$-12$, and for those two relatively large systems, $\epsilon = 1$e$-11$.

\begin{center}
\begin{table}[!htbp]
\caption{The numerical results for systems obtained by gradient type methods with different step size strategies.}
\label{t1}
\begin{center}
{\small
\begin{tabular}{|c| c c c c c|}
\hline
Algorithm & energy (a.u.) & iter &  $\| \nabla_G E\|_F $ & W.C.T (s)&  A.T.P.I (s) \\
\hline
\multicolumn{6}{|c|}{benzene$(C_6H_6) \ \ \  N_g = 102705 \ \ \  N = 15 \ \ \  cores = 8$} \\
\hline
GM-QR-Back  & -3.74246025E+01 & 545  & 9.33E-13 & 24.11 & 0.044 \\
GM-QR-Adap  & -3.74246025E+01 & 334  & 7.53E-13 & 11.36 & 0.034 \\
\hline
\multicolumn{6}{|c|}{aspirin$(C_9H_8O_4) \ \ \  N_g = 133828 \ \ \  N = 34 \ \ \  cores = 16$} \\
\hline
GM-QR-Back  & -1.20214764E+02 & 471  & 9.83E-13 & 43.42 & 0.092 \\
GM-QR-Adap  & -1.20214764E+02 & 327  & 8.86E-13 & 26.47 & 0.081 \\
\hline
\multicolumn{6}{|c|}{fullerene$(C_{60}) \ \ \ N_g = 191805  \ \ \  N = 120 \ \ \  cores = 16$} \\
\hline
GM-QR-Back    & -3.42875137E+02 & 1050 & 9.02E-13 & 945.60 & 0.901 \\
GM-QR-Adap    & -3.42875137E+02 & 558  & 8.17E-13 & 371.26 & 0.665 \\
\hline
\multicolumn{6}{|c|}{alanine chain$(C_{33}H_{11}O_{11}N_{11})\ \ \  N_g = 293725 \ \ \  N = 132
\ \ \  cores = 32$} \\
\hline
GM-QR-Back    & -4.78562217E+02 & 8754  & 9.86E-13 & 10720.76 & 1.225 \\
GM-QR-Adap    & -4.78562217E+02 & 3376  & 9.99E-13 & 3185.21  & 0.943 \\
\hline
\multicolumn{6}{|c|}{carbon nano-tube$(C_{120}) \ \ \ N_g = 354093 \ \ \  N = 240 \ \ \  cores = 32$} \\
\hline
GM-QR-Back    & -6.84467048E+02 & 16161  & 9.99E-13 & 64443.77 & 3.988 \\
GM-QR-Adap    & -6.84467048E+02 & 7929   & 9.98E-13 & 23580.85 & 2.974 \\
\hline
\multicolumn{6}{|c|}{$C_{1015}H_{460} \ \ \ N_g = 1462257 \ \ \  N = 2260 \ \ \  cores = 256$} \\
\hline
GM-QR-Back   & -6.06369982E+03 & 798    & 9.50E-12 & 1764383.74 & 2211.007\\
GM-QR-Adap   & -6.06369982E+03 & 397    & 9.67E-12 & 348390.53  & 877.558 \\
\hline
\multicolumn{6}{|c|}{$C_{1419}H_{556} \ \ \ N_g = 1828847 \ \ \  N = 3116 \ \ \  cores = 512$} \\
\hline
GM-QR-Back    & -8.43085432E+03 & 656    & 9.81E-12 & 3152364.90 & 4805.434 \\
GM-QR-Adap    & -8.43085432E+03 & 368    & 8.26E-12 & 725840.47  & 1972.393 \\
\hline
\end{tabular}}
\end{center}
\end{table}
\end{center}

As is shown in TABLE \ref{t1},  the average computational time for each iteration for our adaptive algorithm GM-QR-Adap is indeed much shorter compared with the backtracking-base algorithms. In addition, our adaptive algorithm needs less iterations to achieve the same accuracy. To see the results more clearly, we present the convergence curves of the residual obtained by gradient type methods with different step size strategies in Fig. \ref{f-1}, from which the similar conclusions can be observed.
\begin{figure}
\centering
\begin{minipage}[b]{0.5 \textwidth}
\centerline{Benzene}
\includegraphics[width=0.9\textwidth]{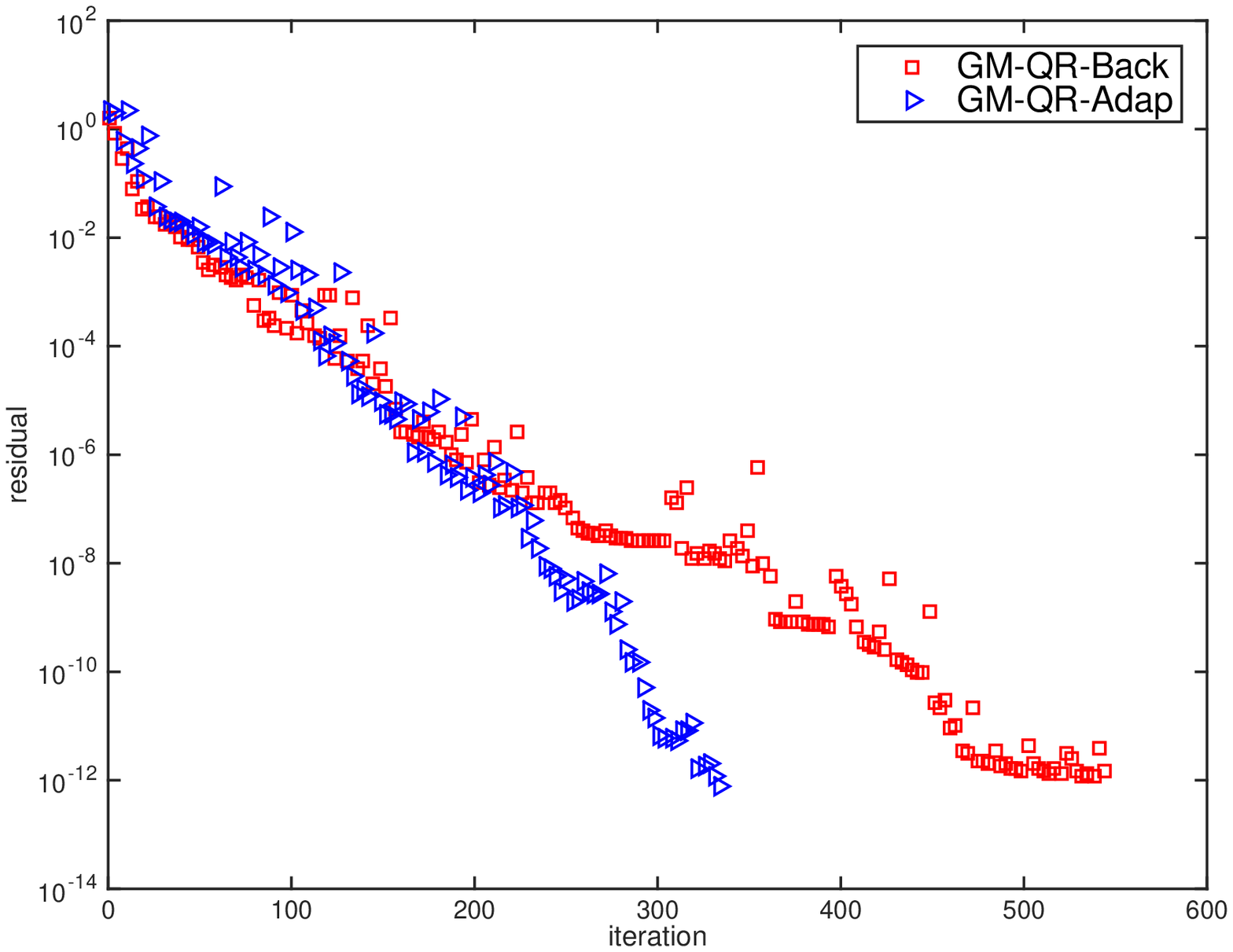}
\end{minipage}
\hspace{-0.2in}
\begin{minipage}[b]{0.5 \textwidth}
\centerline{$C_9H_8O_4$}
\includegraphics[width=0.9\textwidth]{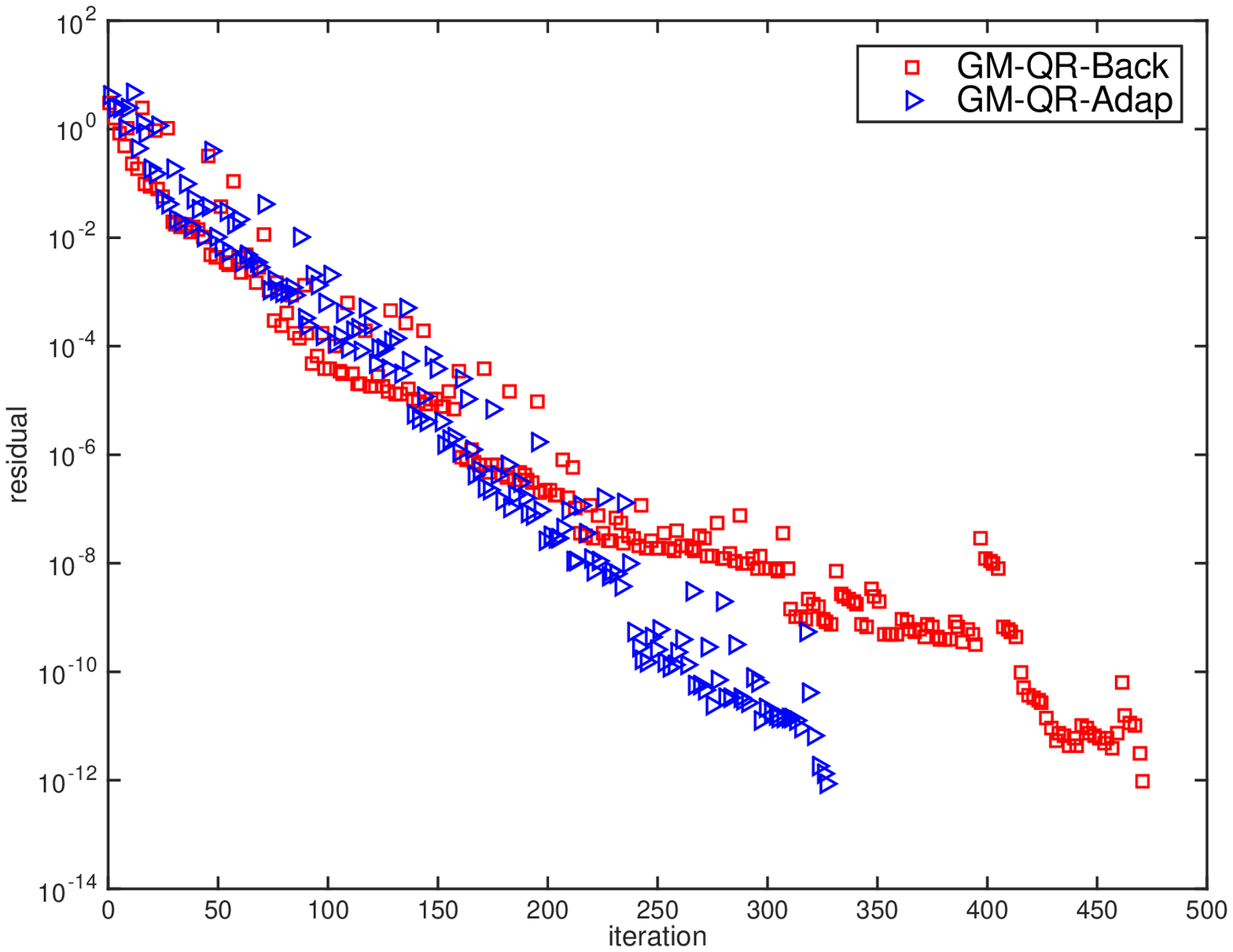}
\end{minipage}

\hspace{-50.0pt}
\par \vspace{-10.pt}
\hspace{-40.0pt}

\begin{minipage}[b]{0.5 \textwidth}
\centerline{$C_{60}$}
\includegraphics[width=0.9\textwidth]{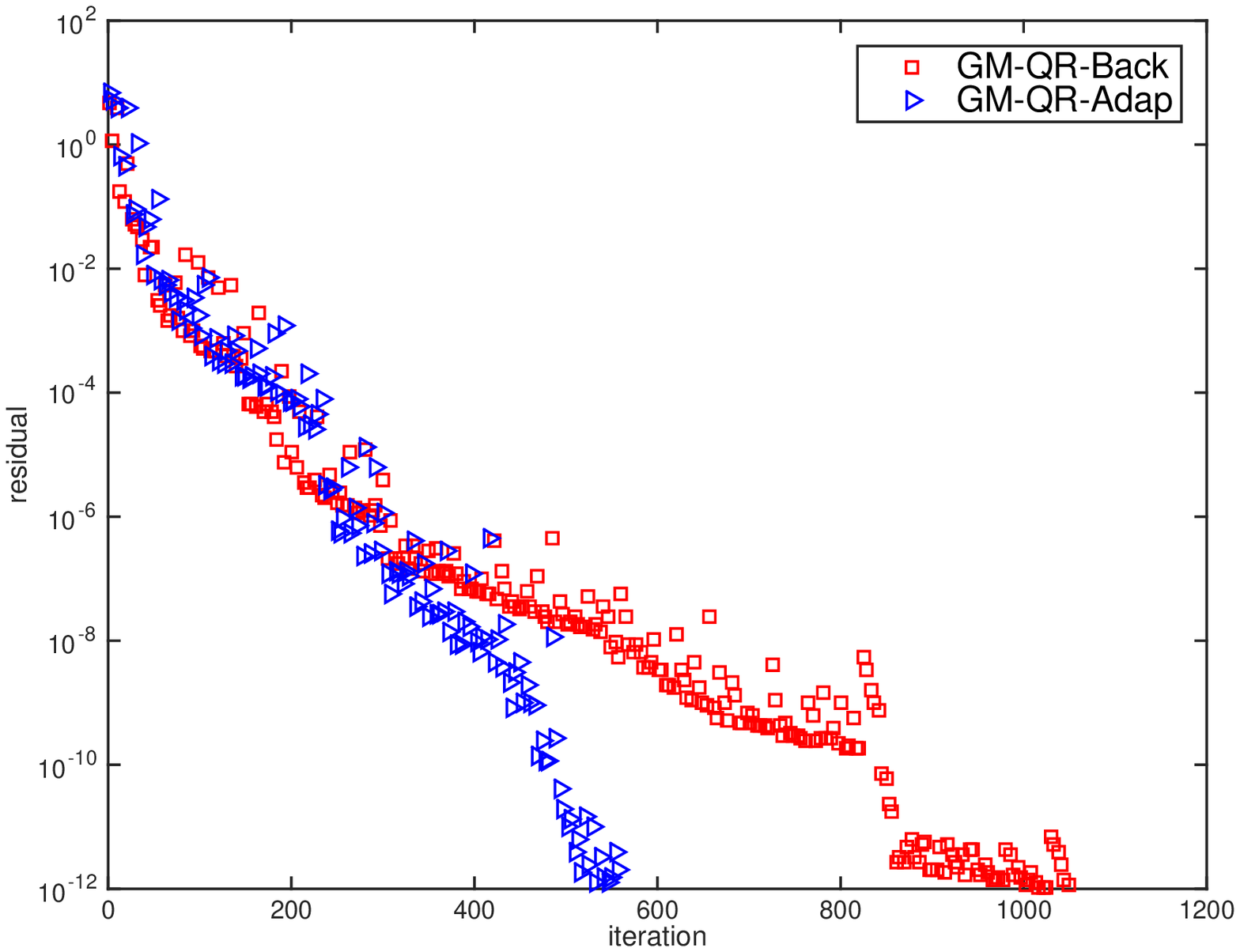}
\end{minipage}
\hspace{-0.2in}
\begin{minipage}[b]{0.5 \textwidth}
\centerline{Alanine}
\includegraphics[width=0.9\textwidth]{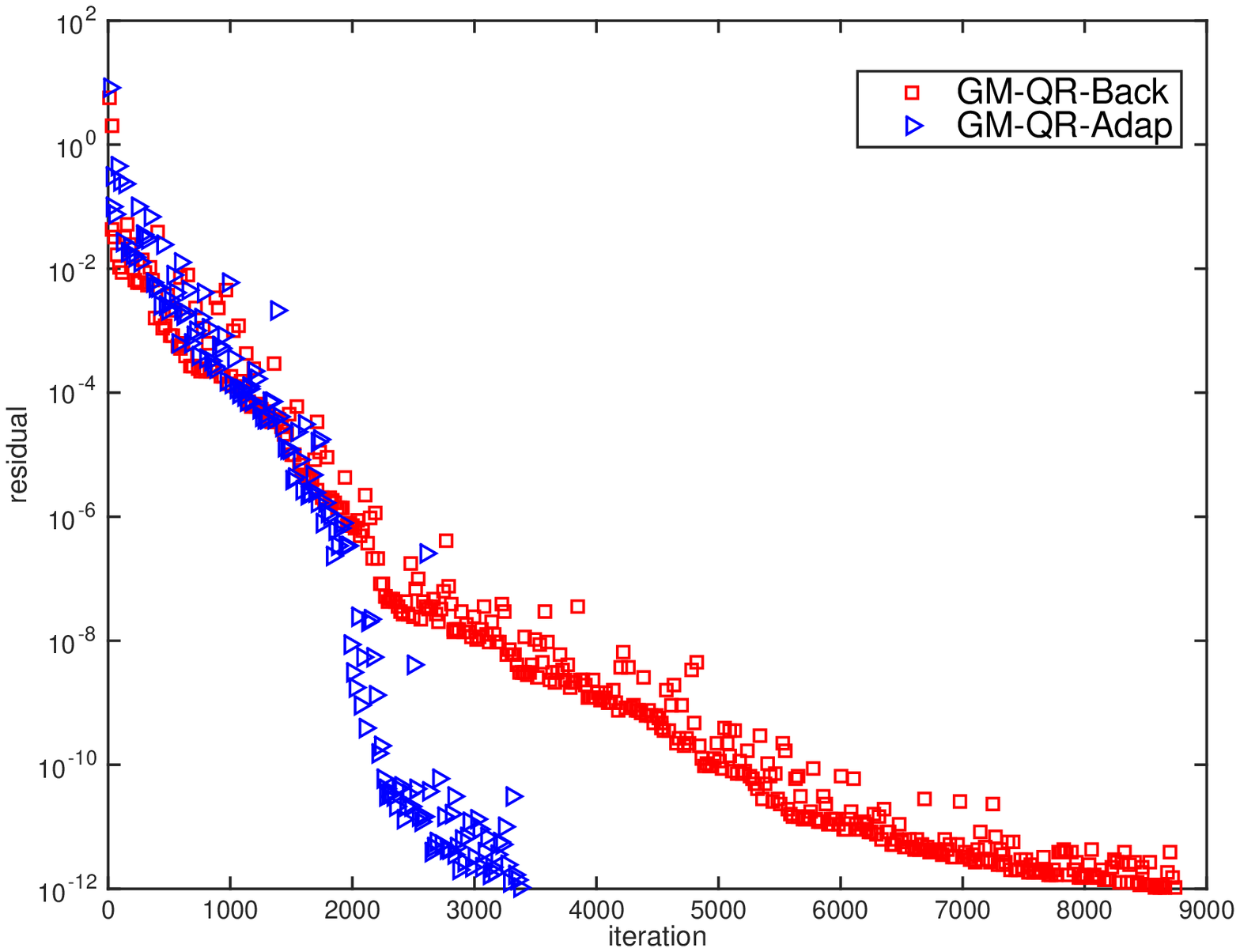}
\end{minipage}

\hspace{-50.0pt}
\par \vspace{-10.pt}
\hspace{-40.0pt}

\begin{minipage}[b]{0.5 \textwidth}
\centerline{$C_{120}$}
\includegraphics[width=0.9\textwidth]{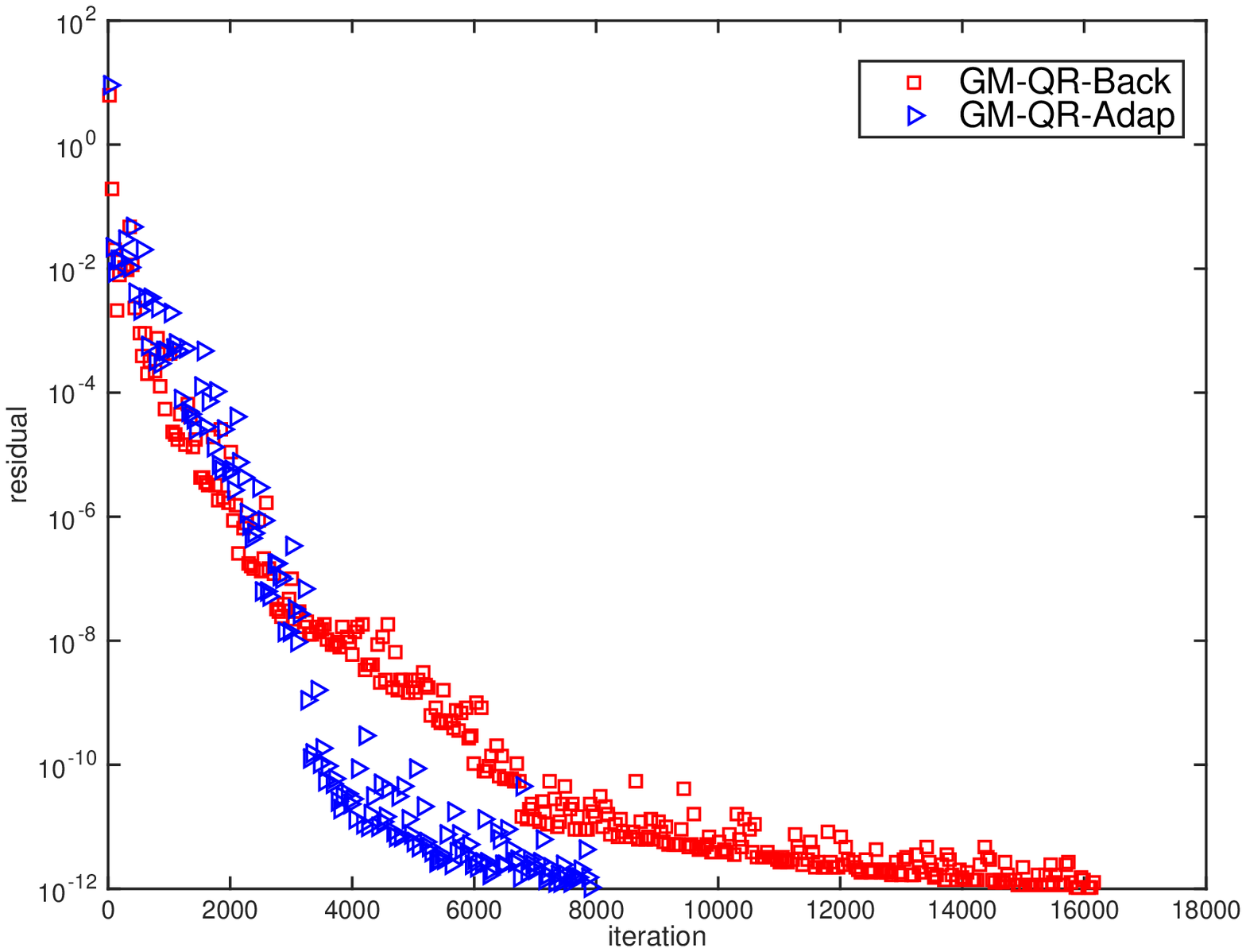}
\end{minipage}
\hspace{-0.2in}
\begin{minipage}[b]{0.5 \textwidth}
\centerline{$C_{1015}H_{460}$}
\includegraphics[width=0.9\textwidth]{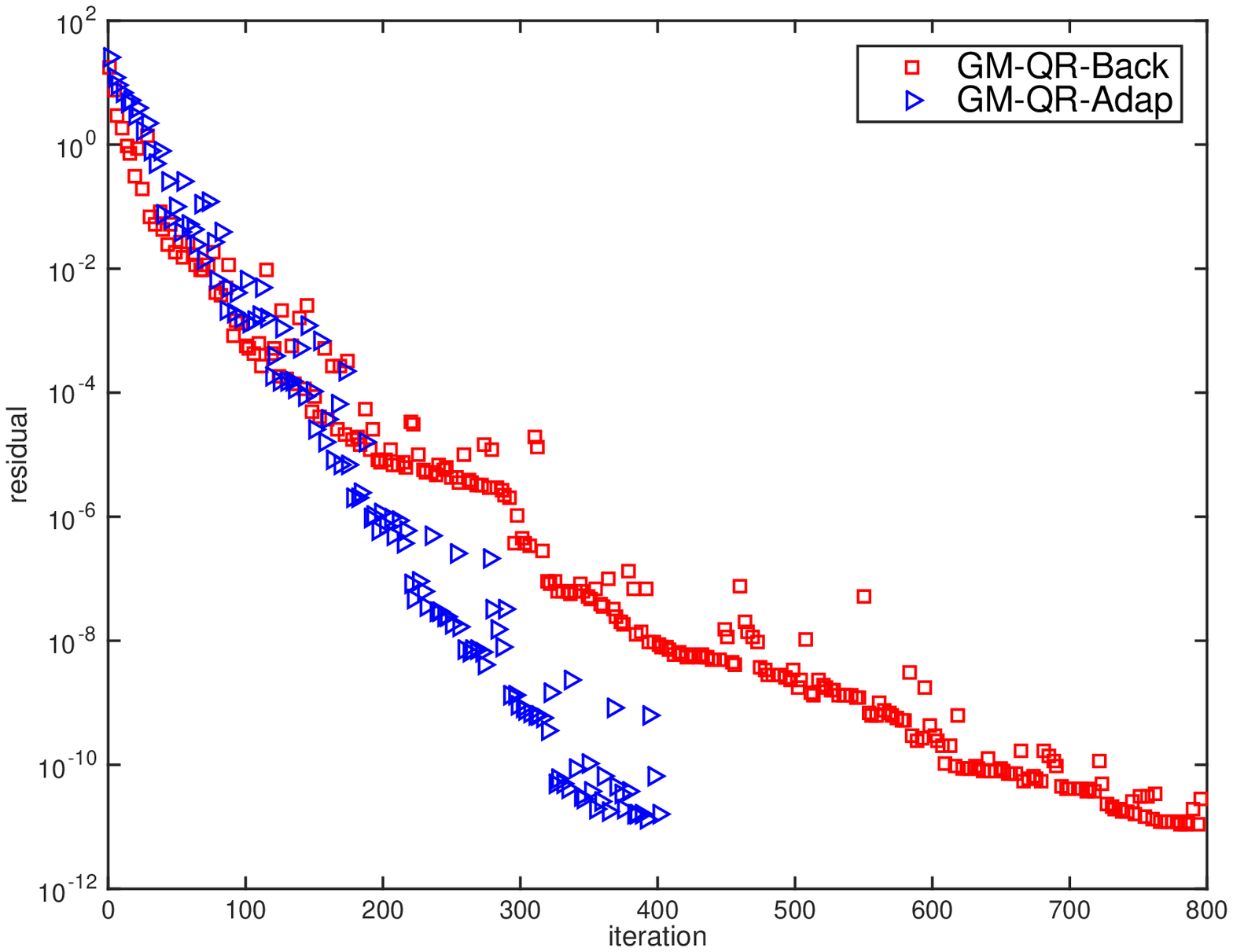}
\end{minipage}

\caption{Convergence curves of $\|\nabla_G E\|_{V^N} $ obtained by different algorithms for different systems.}
\label{f-1}
\end{figure}

We know that in our adaptive algorithm, $\nabla_G^2 E(U_n)[D_n, D_n]$ is calculated once at an iteration, which costs $NN_g^2 + 3N^2N_g + N^3$ flops while calculating $$U_{n+1} = \textup{ortho}(U_n,D_n,t_n),$$ and the corresponding $E(U_{n+1})$ needs $(2N+1)N_g^2 + (7N^2+2)N_g + O(N^3)$ flops \cite{JD} which is the main part in our computation. In Fig. \ref{f0-1} and Fig. \ref{f0-2}, we take $C_{120}$ and $C_{1015}H_{460}$ as examples to see the relationship between computational time per iteration and the number of backtracking steps at each iteration for GM-QR-Back.
\begin{figure}
\centering
\caption{relationship between computational time per step and number of backtracking steps for GM-QR-Back (for $C_{120}$).}
\includegraphics[width=0.9\textwidth]{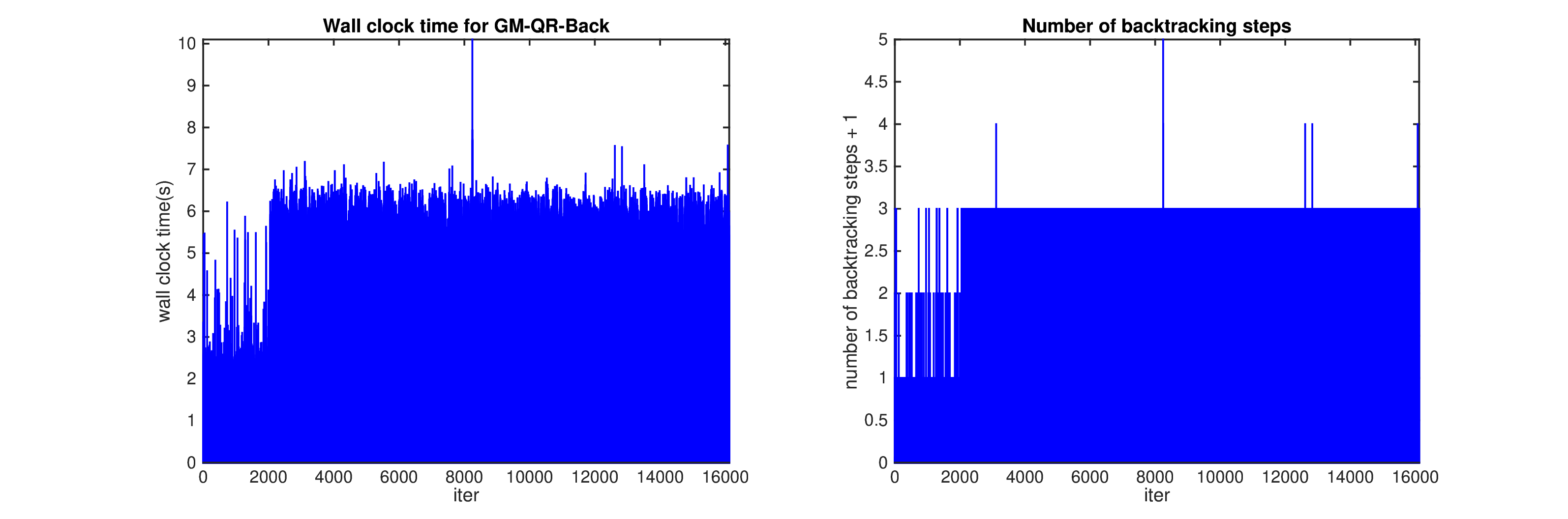} \\
\label{f0-1}
\end{figure}

\begin{figure}
\centering
\caption{relationship between computational time per step and number of backtracking steps for GM-QR-Back (for $C_{1015}H_{460}$).}
\includegraphics[width=0.9\textwidth]{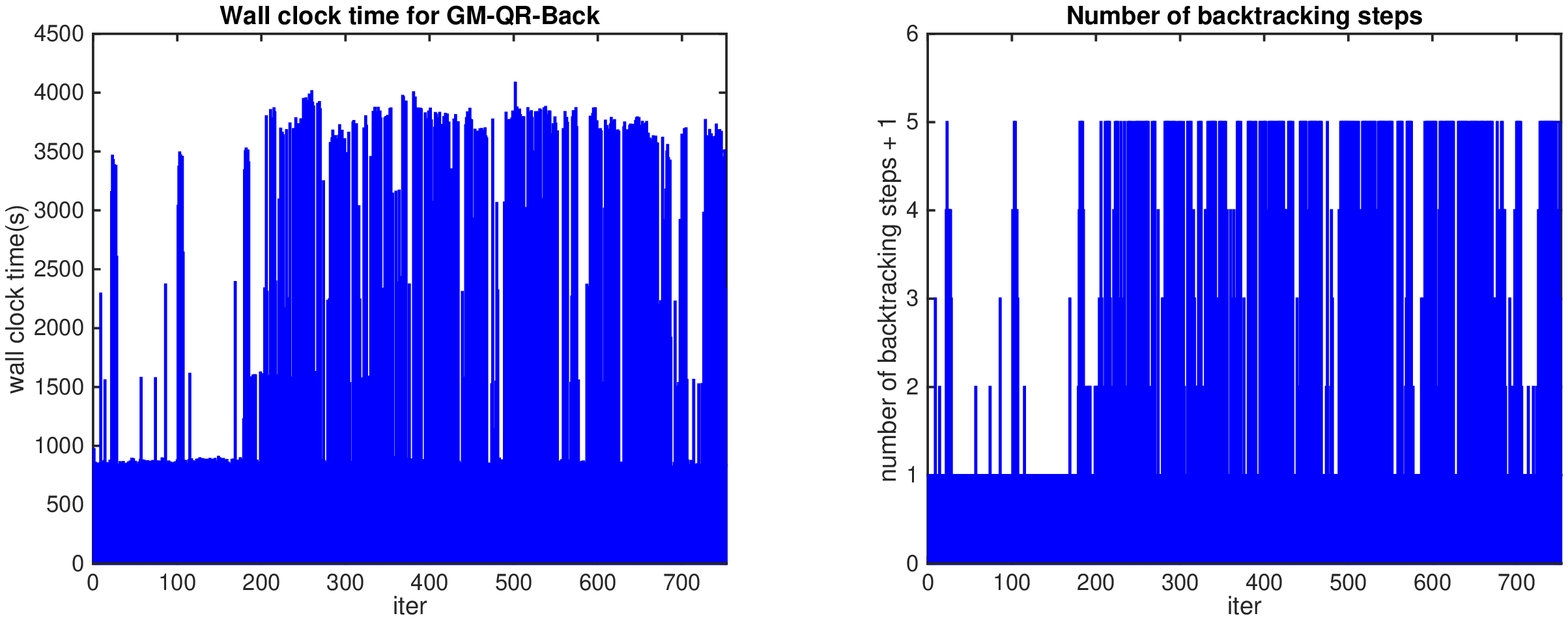} \\
\label{f0-2}
\end{figure}

As is shown in Fig. \ref{f0-1} and Fig. \ref{f0-2}, the trend of the computational time is almost the same as the change of number of backtracking steps at each iteration, which is consistent to what we predicted previously. This phenomenon shows that the orthogonalization procedure and the computation of the objective functional value are the main part of our computation.

For comparison, we show the CPU time required by GM-QR-Back and GM-QR-Adap at each step for $C_{120}$ and $C_{1015}H_{460}$ in Fig. \ref{f1} and Fig. \ref{f2} respectively.
\begin{figure}
\centering
\caption{Computational time per iter. for $C_{120}$ obtained by different algorithms.}
\hspace{12em}
\includegraphics[width=0.9\textwidth]{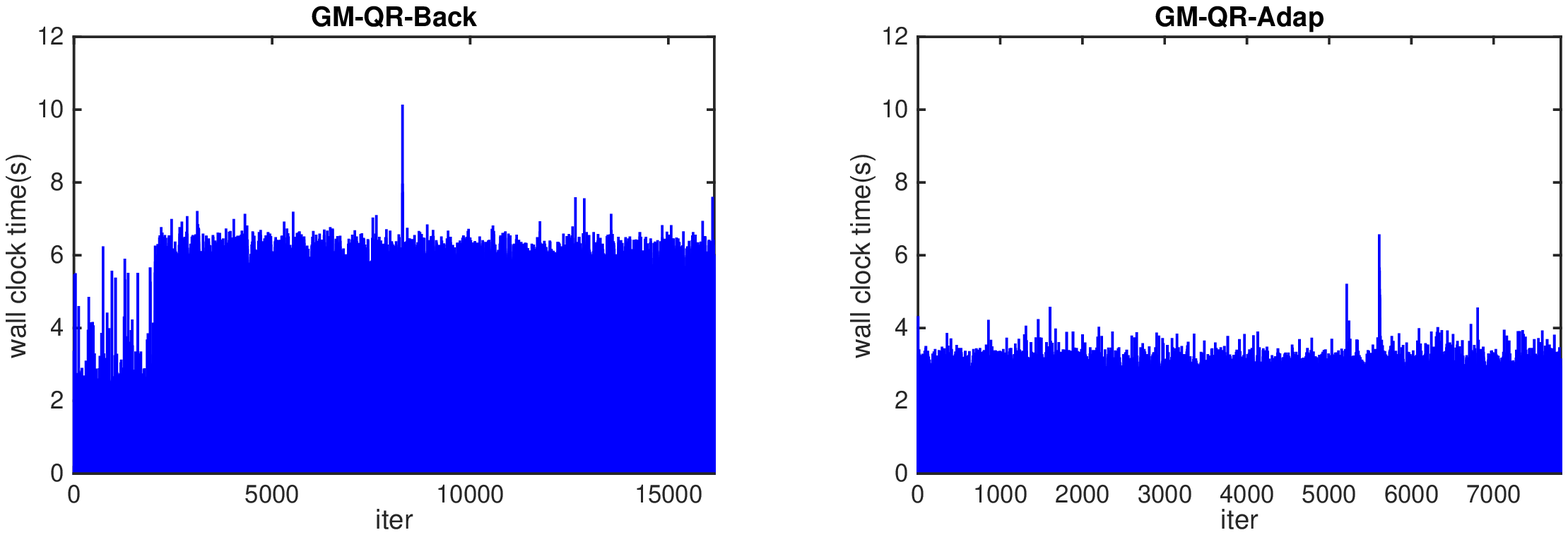} \\
\label{f1}
\end{figure}

\begin{figure}
\centering
\caption{Computational time per iter. for $C_{1015}H_{460}$ obtained by different algorithms.}
\includegraphics[width=0.9\textwidth]{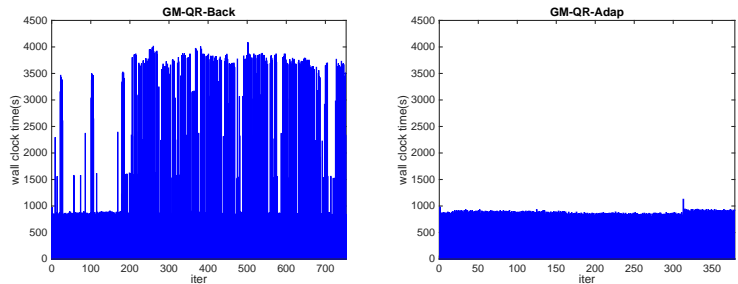} \\
\label{f2}
\end{figure}

It turns out that the computational time spent at each step in our adaptive approach is nearly a constant which approximately equals to the lowest time needed for one step in backtracking-based algorithm, in other words, the computational cost at an iteration, at which the initial step size $t_n^{\textup{initial}}$ does not satisfy \eqref{back-cond1}, reduce significantly by using our adaptive strategy.

We understand that the CG method usually converges faster than the gradient type method. In TABLE \ref{t2}, we compare the numerical results obtained by the gradient type method with our adaptive step size strategy and the CG method for electronic structure calculations(CG-QR) \cite{DLZZ} for the same systems with exactly the same settings as we mentioned before.
\begin{center}
\begin{table}[!htbp]
\caption{The numerical results for systems obtained by different algorithms.}
\label{t2}
\begin{center}
{\small
\begin{tabular}{|c| c c c c c|}
\hline
Algorithm & energy (a.u.) & iter &  $\| \nabla_G E\|_F $ & W.C.T (s)&  A.T.P.I (s) \\
\hline
\multicolumn{6}{|c|}{benzene($C_6H_6) \ \ \  N_g = 102705 \ \ \  N = 15 \ \ \  cores = 8$} \\
\hline
CG-QR        & -3.74246025E+01 & 251  & 9.01E-13 & 12.58 & 0.050 \\
GM-QR-Adap   & -3.74246025E+01 & 334  & 7.53E-13 & 11.36 & 0.034 \\
\hline
\multicolumn{6}{|c|}{aspirin($C_9H_8O_4) \ \ \  N_g = 133828 \ \ \  N = 34 \ \ \  cores = 16$} \\
\hline
CG-QR        & -1.20214764E+02 & 246  & 9.21E-13 & 29.21 & 0.119 \\
GM-QR-Adap   & -1.20214764E+02 & 327  & 8.86E-13 & 26.47 & 0.081 \\
\hline
\multicolumn{6}{|c|}{fullerene$(C_{60}) \ \ \ N_g = 191805  \ \ \  N = 120 \ \ \  cores = 16$} \\
\hline
CG-QR        & -3.42875137E+02 & 391  & 9.45E-13 & 489.00 & 1.251 \\
GM-QR-Adap   & -3.42875137E+02 & 558  & 8.17E-13 & 371.26 & 0.665 \\
\hline
\multicolumn{6}{|c|}{alanine chain$(C_{33}H_{11}O_{11}N_{11})\ \ \  N_g = 293725 \ \ \  N = 132
\ \ \  cores = 32$} \\
\hline
CG-QR        & -4.78562217E+02 & 2100  & 9.98E-13 & 2789.83  & 1.328 \\
GM-QR-Adap   & -4.78562217E+02 & 3376  & 9.99E-13 & 3185.21  & 0.943 \\
\hline
\multicolumn{6}{|c|}{carbon nano-tube$(C_{120}) \ \ \ N_g = 354093 \ \ \  N = 240 \ \ \  cores = 32$} \\
\hline
CG-QR        & -6.84467048E+02 & 3517   & 9.90E-13 & 12976.96 & 3.690 \\
GM-QR-Adap   & -6.84467048E+02 & 7929   & 9.98E-13 & 23580.85 & 2.974 \\
\hline
\multicolumn{6}{|c|}{$C_{1015}H_{460} \ \ \ N_g = 1462257 \ \ \  N = 2260 \ \ \  cores = 256$} \\
\hline
CG-QR        & -6.06369982E+03 & 266    & 9.17E-12 & 299047.84  & 1124.237\\
GM-QR-Adap   & -6.06369982E+03 & 397    & 9.67E-12 & 348390.53  & 877.558 \\
\hline
\multicolumn{6}{|c|}{$C_{1419}H_{556} \ \ \ N_g = 1828847 \ \ \  N = 3116 \ \ \  cores = 512$} \\
\hline
CG-QR        & -8.43085432E+03 & 272    & 9.71E-12 & 722678.98  & 2656.908 \\
GM-QR-Adap   & -8.43085432E+03 & 368    & 8.26E-12 & 725840.47  & 1972.393 \\
\hline
\end{tabular}}
\end{center}
\end{table}
\end{center}

As is shown in TABLE \ref{t2}, though CG-QR method needs less iterations to converge, our adaptive strategy enables the gradient type method to be comparable as CG method in computational time.
\begin{remark}
When performing CG-QR method in numerical experiments, the backtracking step is skipped. The authors in \cite{DLZZ} mentioned that a lack of backtracking may not influence the convergence numerically. After studying the step size strategy therein, we find that the initial guess of the step size $t_n^{\textup{initial}}$ used in \cite{DLZZ} is ``acceptable" in our discussion, i.e., it satisfies $\zeta_n(t_n^{\textup{initial}})>\eta$ and $t_n^{\textup{initial}}\|D_n\|\leq\vartheta_n$ when the parameters are chosen properly. This may explain the reason why the backtracking procedure can be neglected in \cite{DLZZ}. In addition, it has also been reported by the numerical experiments in \cite{DLZZ} that the gradient type method with step size \eqref{stepsize} at every iteration performs relatively bad.
\end{remark}

Consequently, we may conclude that our adaptive step size strategy can not only reduce the cost at each iteration but also accelerate the convergence of an orthogonality constrained line search method. In particular, it enables the gradient type method to be somehow comparable to the CG method, which provides an alternative way to solve an orthogonality constrained minimization problem efficiently.

\section{Concluding remarks}\label{sec-cln}
In this paper, we have set up an uniform approach for a class of line search methods for orthogonality constrained problems.
In particular, we have proposed an adaptive step sizes strategy that can reduce the
cost of choosing suitable step sizes.
We have also proved the convergence of the adaptive line search methods. 
As an application, we apply our method and strategy to solve the Kohn-Sham energy minimization problem. 
The numerical experiments show
that our adaptive approach performs better when compared with the classic backtracking-based algorithm.

Although we have applied our algorithm to electronic structure
calculations only, we believe that our adaptive strategy is applicable to other manifold
constrained problems as long as the cost of computing the retraction is expensive.
In further, our adaptive strategy can be of course incorporated into other line
search methods, for example, the algorithm based on an Armijo-type
condition in \cite{JD}. Besides, despite that we only choose the gradient type method as an example 
in our numerical experiments, it is a straightforward idea to use our adaptive step 
size strategy to other line search methods with different search directions as long as the backtracking 
step occurs frequently or is expensive.

We should emphasize that the objective function is required to be of second order
derivable to compute the estimator in our adaptive algorithm. This requirement
may be too strong in some cases for which other kinds of estimators are
demanded. Note also that in our numerical experiments, $\{\vartheta_n\}_{n=0}^\infty$ are chosen to be
a fixed number. 
There may be some better ways to determine $\{\vartheta_n\}_{n=0}^\infty$ which remains under investigation.

\section*{Acknowledgements} The authors would like to thank  Professor Xin Liu for his comments and suggestions
that improve the presentation of this paper.

\Appendix
\renewcommand{\appendixname}{Appendix~\Alph{section}}
\section{Proof and remarks of Theorem \ref{btbb}}\label{A1}

{\bf Proof of Theorem \ref{btbb}:} We only need to show that
\begin{equation}
\liminf_{n\to\infty}t_n\neq0,
\end{equation}
or else, there exists a subsequence $\{t_{n_s}\}_{s=0}^{\infty}$ such that $\lim_{s\to\infty}t_{n_s}=0$ and $\frac{t_{n_s}}{k}$ does not satisfy \eqref{Armijo}, in other words,
$$E(\textup{ortho}(U_{n_s},D_{n_s},\frac{t_{n_s}}{k}))-\mathcal{C}_{n_s}>\eta\frac{t_{n_s}}{k}\langle\nabla_G E(U_{n_s}), D_{n_s}\rangle_{V^N}.$$
It has been computed in \eqref{non-mono} that
\begin{eqnarray*}
E(U_{n_s})-\mathcal{C}_{n_s} &=& \frac{Q_n-1}{Q_n}\big(E(U_{n_s})-\mathcal{C}_{n_s-1}\big) \\
&\leq& \frac{Q_n-1}{Q_n}\eta t_{n_s-1}\langle\nabla E(U_{n_s-1}), D_{n_s-1}\rangle_{V^N}<0,
\end{eqnarray*}
which leads to
$$E(\textup{ortho}(U_{n_s},D_{n_s},\frac{t_{n_s}}{k}))-E(U_{n_s})>\eta\frac{t_{n_s}}{k}\langle\nabla_G E(U_{n_s}), D_{n_s}\rangle_{V^N}.$$
A simple calculation gives that
\begin{equation}\label{contra2}
\frac{E(\textup{ortho}(U_{n_s},D_{n_s},\frac{t_{n_s}}{k}))-E(U_{n_s})}{\frac{t_{n_s}}{k}}>\eta\langle\nabla_G E(U_{n_s}), D_{n_s}\rangle_{V^N}.
\end{equation}
For simplicity, we again denote $E(\textup{ortho}(U_{n_s},D_{n_s},t))$ by $\phi_{n_s}(t)$ as an function of $t$, then
$$\phi'_{n_s}(t) = \langle\nabla E(\textup{ortho}(U_{n_s},D_{n_s},t)), \dot{\textup{ortho}}(U_{n_s},D_{n_s},t)\rangle_{V^N} .$$
\eqref{contra2} indicates that there exists an $\xi_{n_s}\in(0,\frac{t_{n_s}}{k})$ such that
\begin{eqnarray*}
\phi'_{n_s}(\xi_{n_s}) >\eta\langle\nabla_G E(U_{n_s}), D_{n_s}\rangle_{V^N}.
\end{eqnarray*}
Since $\mathcal{M}^N$ and the set $\{D\in\mathcal{T}\mathcal{G}_N: \|D\|_{V^N}\le C\}$ are both compact, Assumption \ref{Dbound} indicates that there exists a subsequence of $\{n_s\}_{s=0}^{\infty}$ which is also denoted by $\{n_s\}_{s=0}^{\infty}$ without loss of generality, such that
$U_{n_s}\to \tilde{U}$ and $D_{n_s}\to\tilde{D}$
for some $\tilde{U}\in\mathcal{M}^N$ and $\tilde{D}\in\mathcal{T}\mathcal{G}_N$ as $s\to\infty$. Moreover, We see that $$\tilde{U}^T\tilde{D} = \lim_{s\to\infty}U_{n_s}^TD_{n_s} = 0,$$ and hence, $\tilde{D}\in\mathcal{T}_{[\tilde{U}]}\mathcal{G}_N$.

Due to {$\displaystyle\lim_{s\to\infty}\xi_{n_s}=0$}, we have
$$\lim_{s\to\infty}\phi'_{n_s}(\xi_{n_s}) = \langle\nabla E(\textup{ortho}(\tilde{U},\tilde{D},0)), \dot{\textup{ortho}}(\tilde{U},\tilde{D},0)\rangle_{V^N}\ge\eta\langle\nabla_G E(\tilde{U}), \tilde{D}\rangle_{V^N},$$
which combining with \eqref{retraction1} and \eqref{retraction2} gives that
$$\langle\nabla E(\tilde{U}), \tilde{D}\rangle_{V^N}\ge\eta\langle\nabla_G E(\tilde{U}), \tilde{D}\rangle_{V^N}.$$
Note that $\langle\nabla E(\tilde{U}), \tilde{D}\rangle_{V^N}=\langle\nabla_G E(\tilde{U}), \tilde{D}\rangle_{V^N}$ and $1-\eta>0$, we have
$$\langle\nabla_G E(\tilde{U}), \tilde{D}\rangle_{V^N} \ge 0.$$
As a result,
$$0\geq\lim_{s\to\infty}\langle\nabla_G E(U_{n_s}), D_{n_s}\rangle_{V^N} =\langle\nabla_G E(\tilde{U}), \tilde{D}\rangle_{V^N} \ge 0.$$
We obtain from \eqref{not-ortho} that
$$\lim_{s\to\infty}\|\nabla_G E(U_{n_s})\|_{V^N}^a=0,$$
which completes the proof.
\begin{remark}
The search directions $\{D_n\}_{n=0}^\infty$ satisfying \eqref{decrease-direction}, \eqref{not-ortho} and Assumption \ref{Dbound} are called ``gradient related" in \cite{AbMaSe}. We use the similar approach and extend the convergence result therein to the ``non-monotone" case. The similar result can also be found in \cite{HMWY}, but the search directions in \cite{HMWY} are fixed to be the negative gradient directions.
\end{remark}

Due to Theorem \ref{btbb}, we are able to obtain the convergence results of some existing methods under weaker assumptions. For instance, we have

\begin{itemize}
\item The gradient
type method proposed in \cite{ZZWZ} will eventually give a stationary point as long as
the gradient $\nabla E$ of the objective function is bounded.
The original result was established based on the assumption that
$\nabla E$ is Lipschitz continuous.

\item If we restart the CG method
proposed in \cite{DLZZ} periodically(or restart the algorithm when $\frac{\langle\nabla_G E(U_{n_s}), D_{n_s}\rangle_{V^N}}{\|\nabla_G E(U_n)\|_{V^N}^a}<\delta$, where $\delta>0$ is a given parameter), then the algorithm globally converges to a stationary point
for all kinds of retractions provided that $\nabla E$ is bounded.
For comparison, the original result only works for 3 particular retractions and need to assume that
$\nabla E$ is Lipschitz continuous and the Hessian of the objective function is positive defined around the
stationary point, and as a result, is a local convergence.

We point out that a restarted version of the CG method is also suggested in \cite{DLZZ} with a different restarted strategy.
\end{itemize}

\section{Detailed results of backtracking-based gradient method}\label{A2}

In this appendix, we provide some numerical results obtained by the gradient type method with different initial step size choices and different parameters to illustrate the reason why we choose the results showed in Section \ref{sec-num} for comparison and to motivate our adaptive step size strategy clearer.

As we have mentioned, the initial step size $t_n^{\textup{initial}}$ can be chosen as \eqref{tau1}, \eqref{tau2}, or \eqref{initialstep}.  We test these three cases for some small systems to determine which one to be used and denote them by GM-QR-Back-odd, GM-QR-Back-even and GM-QR-Back, respectively.  The detailed results are shown in TABLE \ref{t3}.

\begin{center}
\begin{table}[!htbp]
\caption{The numerical results obtained by backtracking-based gradient type methods with different initial step sizes.}
\label{t3}
\begin{center}
{\small
\begin{tabular}{|c| c c c c c|}
\hline
Algorithm & energy (a.u.) & iter &  $\| \nabla_G E\|_F $ & W.C.T (s)&  A.T.P.I (s) \\
\hline
\multicolumn{6}{|c|}{benzene$(C_6H_6) \ \ \  N_g = 102705 \ \ \  N = 15 \ \ \  cores = 8$} \\
\hline
GM-QR-Back-odd    & -3.74246025E+01 & 625  & 9.48E-13 & 31.39 & 0.050 \\
GM-QR-Back-even  & -3.74246025E+01 & 850  & 7.53E-13 & 42.07 & 0.049 \\
GM-QR-Back           & -3.74246025E+01 & 545  & 9.33E-13 & 24.11 & 0.044 \\
\hline
\multicolumn{6}{|c|}{aspirin$(C_9H_8O_4) \ \ \  N_g = 133828 \ \ \  N = 34 \ \ \  cores = 16$} \\
\hline
GM-QR-Back-odd    & -1.20214764E+02 & 609  & 9.98E-13 & 61.00 & 0.100 \\
GM-QR-Back-even  & -1.20214764E+02 & 583  & 3.67E-13 & 58.73 & 0.101 \\
GM-QR-Back           & -1.20214764E+02 & 471  & 9.83E-13 & 43.42 & 0.092 \\
\hline
\multicolumn{6}{|c|}{fullerene$(C_{60}) \ \ \ N_g = 191805  \ \ \  N = 120 \ \ \  cores = 16$} \\
\hline
GM-QR-Back-odd      & -3.42875137E+02 & 6597 & 9.99E-13 &6775.76 & 1.027 \\
GM-QR-Back-even    & -3.42875137E+02 & 1325 & 9.74E-13 &1205.51 & 0.910 \\
GM-QR-Back             & -3.42875137E+02 & 1050 & 9.02E-13 & 945.60 & 0.901 \\
\hline
\end{tabular}}
\end{center}
\end{table}
\end{center}

Hence, we choose $t_n^{\textup{initial}}$ as \eqref{initialstep} for both backtracking-based method and adaptive step size based method.

Besides, to confirm the truth that backtracking procedure is necessary for backtra- cking-based 
gradient type method, we test the case in which no backtracking is imposed, i.e., we choose $t_n = t_n^{\textup{initial}} $ (or, by setting $\mathcal{C}_n=+\infty$) for all $n$ and show the corresponding results in TABLE \ref{t4} in which ``GM-QR-noBack" denotes the gradient type method with $t_n = t_n^{\textup{initial}}$.

\begin{center}
\begin{table}[!htbp]
\caption{The numerical results for systems obtained by different algorithms.}
\label{t4}
\begin{center}
{\small
\begin{tabular}{|c| c c c c c|}
\hline
Algorithm & energy (a.u.) & iter &  $\| \nabla_G E\|_F $ & W.C.T (s)&  A.T.P.I (s) \\
\hline
\multicolumn{6}{|c|}{benzene($C_6H_6) \ \ \  N_g = 102705 \ \ \  N = 15 \ \ \  cores = 8$} \\
\hline
GM-QR-noBack        & -3.74246025E+01 & 445  & 9.91E-13 & 10.85 & 0.025 \\
GM-QR-Adap   & -3.74246025E+01 & 334  & 7.53E-13 & 11.36 & 0.034 \\
\hline
\multicolumn{6}{|c|}{aspirin($C_9H_8O_4) \ \ \  N_g = 133828 \ \ \  N = 34 \ \ \  cores = 16$} \\
\hline
GM-QR-noBack        & -1.20214764E+02 & 372  & 9.87E-13 & 22.66 & 0.061 \\
GM-QR-Adap   & -1.20214764E+02 & 327  & 8.86E-13 & 26.47 & 0.081 \\
\hline
\multicolumn{6}{|c|}{fullerene$(C_{60}) \ \ \ N_g = 191805  \ \ \  N = 120 \ \ \  cores = 16$} \\
\hline
GM-QR-noBack        & -3.42875137E+02 & 1442  & 9.99E-13 & 781.82 & 0.542 \\
GM-QR-Adap   & -3.42875137E+02 & 558  & 8.17E-13 & 371.26 & 0.665 \\
\hline
\multicolumn{6}{|c|}{alanine chain$(C_{33}H_{11}O_{11}N_{11})\ \ \  N_g = 293725 \ \ \  N = 132
\ \ \  cores = 32$} \\
\hline
GM-QR-noBack        & -4.78562217E+02 & 30000  & 1.34E-12 & 21371.76  & 0.712 \\
GM-QR-Adap   & -4.78562217E+02 & 3376  & 9.99E-13 & 3185.21  & 0.943 \\
\hline
\multicolumn{6}{|c|}{carbon nano-tube$(C_{120}) \ \ \ N_g = 354093 \ \ \  N = 240 \ \ \  cores = 32$} \\
\hline
GM-QR-noBack        & -6.84466094E+02 & 30000   & 3.67E-03 & 70283.80 & 2.342 \\
GM-QR-Adap   & -6.84467048E+02 & 7929   & 9.98E-13 & 23580.85 & 2.974 \\
\hline
\end{tabular}}
\end{center}
\end{table}
\end{center}

We see from TABLE \ref{t4} that though the computational time at each iteration for a backtracking-free algorithm is lower than our adaptive step size based algorithm, it can not converge within 30000 iterations (the max number of iterations we set) for alanine and $\textup{C}_{120}$. Even though it converges for some small systems, our adaptive step size strategy is still comparable or even performs better.
%
%
%

\end{document}